\RequirePackage{ifpdf}          
\ifpdf 
    \documentclass[10pt,pdftex]{article}
    \usepackage{hyperref} 
    \hypersetup{colorlinks, citecolor=MyDarkRed, filecolor=MyDarkBlue,
      linkcolor=MyDarkBlue, urlcolor=MyDarkBlue}
\fi 

\usepackage[in]{fullpage}       
\usepackage{amsmath,amssymb}	

\usepackage[T1]{fontenc}
\usepackage{tgtermes} 
\usepackage[usenames]{color}   
\definecolor{MyDarkBlue}{rgb}{0, 0.0, 0.45} 
\definecolor{MyDarkRed}{rgb}{0.45, 0.0, 0} 
\definecolor{MyDarkGreen}{rgb}{0, 0.45, 0} 
\definecolor{MyLightGray}{gray}{.90}
\definecolor{MyLightGreen}{rgb}{0.5, 0.99, 0.5}
\DefineNamedColor{named}{Peach}         {cmyk}{0,0.50,0.70,0}
\DefineNamedColor{named}{Cyan}          {cmyk}{1,0,0,0}

\usepackage{graphicx} 
\usepackage{floatflt}
\usepackage{wrapfig}
\usepackage[parfill]{parskip}    
\usepackage{amsfonts, amscd, amssymb, amsthm, amsmath}
\usepackage{mathrsfs}
\usepackage{urwchancal}
\usepackage{xcolor}
\usepackage{MnSymbol}
\usepackage{pdfsync}
\usepackage{enumerate}
\usepackage[mathscr]{euscript}
\theoremstyle{plain}
\newtheorem{thm}{Theorem}[section]
\newtheorem*{theorem-non}{Theorem}

\newtheorem{prop}[thm]{Proposition}
\newtheorem{cor}[thm]{Corollary}

\theoremstyle{definition}

\theoremstyle{remark}
\newtheorem*{rem}{Remark}

\usepackage{color}
\definecolor{dred}{rgb}{.65, 0, 0.15}

\def\<{\langle} \def\>{\rangle}

\def\dim{\mathrm{dim}}

\usepackage{listings}           
\pagecolor{white}		

\begin{document}

\title{Equivariant Euler characteristics of $\overline{\mathscr{M}}_{g,  n}$}	
\author{Adrian Diaconu \\         
}
\date{}	                
\maketitle
\begin{abstract} 
\noindent 
Let $\overline{\mathscr{M}}_{g, n}$ be the moduli space of $n$-pointed stable genus $g$ curves, and let 
$\mathscr{M}_{g, n}$ be the moduli space of $n$-pointed smooth curves of genus $g.$ In this paper, 
we obtain an asymptotic expansion for the characteristic of the free modular operad $\mathbb{M}\mathcal{V}$ 
generated by a stable $\mathbb{S}$-module $\mathcal{V},$ allowing to effectively compute $\mathbb{S}_{n}$-equivariant 
Euler characteristics of $\overline{\mathscr{M}}_{g, n}$ in terms of $\mathbb{S}_{n'}$-equivariant Euler characteristics of 
$\mathscr{M}_{g'\!, n'}$ with $0\le g' \le g,$ $\textrm{max}\{0, 3 - 2g' \} \le n' \le 2(g - g') + n.$ This answers a question posed by 
Getzler and Kapranov by making their integral representation of the characteristic of the modular operad $\mathbb{M}\mathcal{V}$ 
effective. To illustrate how the asymptotic expansion is used, we give formulas expressing the generating series of the $\mathbb{S}_{n}$-equivariant 
Euler characteristics of $\overline{\mathscr{M}}_{g, n},$ for $g = 0, 1$ and $2,$ in terms of the corresponding generating series associated with 
$\mathscr{M}_{g, n}.$ 
\end{abstract}

\vskip90pt
\tableofcontents 

\vskip121pt
\line(1,0){100}
\vskip-5pt
\noindent
{\small{School of Mathematics, University of Minnesota, Minneapolis, MN 55455\\
E-mail: cad@umn.edu}}

\newpage
\section{Introduction} 
Let $\mathcal{V}$ be a stable $\mathbb{S}$-module, i.e., a collection of chain complexes 
$\{\mathcal{V}(\!\!(g, n)\!\!)\!\}_{\scriptscriptstyle g, n \, \ge \, 0}$ with an action of the symmetric group $\mathbb{S}_{n}$ 
on $\mathcal{V}(\!\!(g, n)\!\!),$ and such that $\mathcal{V}(\!\!(g, n)\!\!) = 0$ when $2g + n - 2 \le 0,$ and let 
\begin{equation*}
\mathbb{C}\mathrm{h}(\mathcal{V}) \; = \sum_{2(g - 1) + n \, > \, 0} \hbar^{g - 1} \mathrm{ch}(\mathcal{V}(\!\!(g, n)\!\!)\!)
\end{equation*} 
denote the characteristic of $\mathcal{V},$ where $\mathrm{ch}(\mathcal{V}(\!\!(g, n)\!\!)\!)$ is the characteristic of 
the $\mathbb{S}_{n}$-representation $\mathcal{V}(\!\!(g, n)\!\!).$ To the stable $\mathbb{S}$-module $\mathcal{V},$ 
there is the associated free modular operad $\mathbb{M}\mathcal{V}$ generated by $\mathcal{V};$ as such, 
we can also take the characteristic of $\mathbb{M}\mathcal{V}.$  Throughout, we shall consider only $\mathbb{S}$-modules 
in the category of $\ell$-adic Galois representations.

The aim of this paper is to make quite effective a beautiful result of Getzler and Kapranov \cite[Theorem~8.13]{GK}, expressing the 
relationship between the characteristics of $\mathcal{V}$ and $\mathbb{M}\mathcal{V}$ in the form 
\begin{equation} \label{eq: Getz-Kapr1}
\mathbb{C}\mathrm{h}(\mathbb{M}\mathcal{V}) = 
\mathrm{Log}(\exp(\Delta)\mathrm{Exp}(\mathbb{C}\mathrm{h}(\mathcal{V})))
\end{equation} 
where $\Delta$ is a certain analogue of the Laplacian, $\mathrm{Exp}(f)$ is the {\it plethystic} exponential of $f,$ and 
$\mathrm{Log}(f)$ is the inverse of the plethystic exponential; see Section \ref{prelim} for precise definitions. The 
above equality is a consequence of how the functor $\mathbb{M}$ on the category of stable 
$\mathbb{S}$-modules is tailored to the way the boundary strata of $\overline{\mathscr{M}}_{g, n}$ are obtained 
by gluing together moduli spaces $\mathscr{M}_{g'\!, n'}.$ 

The formula \eqref{eq: Getz-Kapr1} is a natural generalization of Wick's theorem \cite{BIZ}, 
which gives the integral formula
\begin{equation} \label{eq: Wick} 
\sum_{2(g - 1) + n \, > \, 0} \mathrm{M}v_{g, n}\, \hbar^{g - 1} \frac{\xi^{n}}{n!}
= \, \log \!\int_{\mathbb{R}} \, \exp\bigg(\sum_{2(g - 1) + n \, > \, 0} v_{g, n} \,  \hbar^{g - 1} \frac{x^{n}}{n!}
\, - \, \frac{(x - \xi)^{2}}{2 \hbar}\bigg)\, 
\frac{dx}{\sqrt{2 \pi \hbar}}.
\end{equation} 
Here $\{v_{g, n} : 2(g - 1) + n  >  0\}$ is a set of variables, and 
\begin{equation*}
\mathrm{M}v_{g, n} : = \sum_{G\in \mathrm{Ob}\, \Gamma_{\! g, n}} \frac{1}{|\mathrm{Aut}(G)|}
\, \prod_{v \in \mathrm{Vert}(G)} v_{g(v), n(v)} \qquad \text{(for $2 (g - 1) + n >  0$)}
\end{equation*}
where $\Gamma_{\! g, n}$ is the {\it finite} category whose objects are isomorphism classes of stable graphs of genus 
$g$ with $n$ ordered legs, and whose morphisms are the automorphisms; see \cite{GK} for details. Thus if we define
\begin{equation*}
b_{\!\scriptscriptstyle g} = b_{\!\scriptscriptstyle g}(\xi) \; =  
\sum_{n \, \ge \, \textrm{max}\{0,\, 3 - 2g\}} \mathrm{M}v_{g, n}\, \frac{\xi^{n}}{n!} \;\;\; \mathrm{and} \;\;\;
a_{\scriptscriptstyle g} = a_{\scriptscriptstyle g}(x) \; =  
\sum_{n \, \ge \, \textrm{max}\{0,\, 3 - 2g\}} v_{g, n}\, \frac{x^{n}}{n!}
\end{equation*} 
then, \!by performing an asymptotic expansion of the integral in the right-hand side of \eqref{eq: Wick}, 
\!one can obtain formulas expressing the coefficients $b_{\!\scriptscriptstyle g}$ in terms of $a_{\scriptscriptstyle g'}$ 
with $g'  \le  g.$

Our main result, Theorem \ref{Main Theorem}, provides an asymptotic expansion of the right-hand side of 
\eqref{eq: Getz-Kapr1}, thus answering a question posed by Getzler and Kapranov in \cite{GK}, p. 113. 
Letting 
$
\mathrm{\bf  a}_{\scriptscriptstyle g}
$ 
and 
$
\mathrm{\bf  b}_{\! \scriptscriptstyle g}
$ 
denote, as above, the coefficients of $\hbar^{g - 1}$ in $\mathbb{C}\mathrm{h}(\mathcal{V})$ and 
$\mathbb{C}\mathrm{h}(\mathbb{M}\mathcal{V}),$ respectively, then as consequences of 
Theorem \ref{Main Theorem}, we shall obtain formulas for 
$
\mathrm{\bf  b}_{\scriptscriptstyle 0}, \mathrm{\bf  b}_{\scriptscriptstyle 1}
$ 
\!and 
$
\mathrm{\bf  b}_{\scriptscriptstyle 2}
$ 
in terms of  
$
\mathrm{\bf  a}_{\scriptscriptstyle 0}, \mathrm{\bf  a}_{\scriptscriptstyle 1}
$ 
\!and 
$
\mathrm{\bf  a}_{\scriptscriptstyle 2},
$ 
see Section \ref{equid-euler-examples}. The formulas for 
$
\mathrm{\bf  b}_{\scriptscriptstyle 0}
$ 
and 
$
\mathrm{\bf  b}_{\scriptscriptstyle 1} 
$ 
(Theorem \ref{Getzler1} and Theorem \ref{Getzler2}, respectively) are not new, see \cite[Theorem~7.17]{GK} or \cite[Theorem~5.9]{Getz1} for the calculation of 
$
\mathrm{\bf  b}_{\scriptscriptstyle 0},
$ 
and \cite{Getz2} or \cite{Peter1} for that of 
$
\mathrm{\bf  b}_{\scriptscriptstyle 1}.
$ 
The proofs of these two results were merely included as examples of how the coefficients 
$
\mathrm{\bf  b}_{\! \scriptscriptstyle g}
$ 
are calculated.

The argument used to prove Theorem \ref{Main Theorem} represents the natural generalization of the method 
used by Bini and Harer \cite[Section~3]{BH} to study the asymptotic expansion of the integral in \eqref{eq: Wick}. 
It is also possible to give an analogous interpretation of the asymptotic expansion discussed in this paper to that in 
\cite[Proposition~3.6]{BH}, as an expansion over stable graphs, to obtain explicit formulas for the coefficients 
$
\mathrm{\bf  b}_{\! \scriptscriptstyle g}
$ 
when $g \ge 3.$ The general formula for 
$
\mathrm{\bf  b}_{\! \scriptscriptstyle g}
$ 
can then be combined with a result of Gorsky \cite{Gors}, where the author establishes a formula for the 
generating series of the numerical $\mathbb{S}_{n}$-equivariant Euler characteristics of $\mathscr{M}_{g, n},$ 
to obtain the corresponding numerical $\mathbb{S}_{n}$-equivariant characteristics of 
$\overline{\mathscr{M}}_{g, n}.$

Although we shall work throughout just in the tensor symmetric abelian category 
$
\mathbf{Rep}_{\mathbb{Q}_{\scriptscriptstyle \ell}}(\mathrm{Gal}(\overline{\mathbb{Q}}\slash \mathbb{Q})) 
$ 
of $\ell$-adic Galois representations,\!\footnote{In fact, we shall only deal with virtual $\ell$-adic Galois representations (Euler characteristics); 
thus it will suffice to work instead with the characters of the representations at Frobenius elements.} 
the results in this paper are valid for stable $\mathbb{S}$-modules in any symmetric monoidal category 
with finite colimits and additive over a field of characteristic zero.

\vskip5pt
{\bf Acknowledgements.} I am grateful to Jonas Bergstr\"om and Dan Petersen for their comments and suggestions. 
I would also like to thank them for making available to me the formulas for the $\mathbb{S}_{n}$-equivariant Euler characteristics of $\overline{\mathscr{M}}_{2, n}$ for small values of $n.$

\section{Notation and preliminaries} \label{prelim}
{\bf Symmetric functions.} Let $\mathbb{S}_{k}$ denote the symmetric group on $k$ letters. 
The completed ring of symmetric functions in infinitely many variables (see \cite{GK}) is defined by    
\begin{equation*}
\Lambda = \varprojlim \mathbb{Z}[\!\![x_{\scriptscriptstyle 1}, \ldots, x_{\scriptscriptstyle k}]\!\!]^{\mathbb{S}_{k}}.
\end{equation*} 
We have the standard functions
\begin{equation*}
e_{\scriptscriptstyle n} \, = \sum_{i_{\scriptscriptstyle 1}  < \, \cdots \, < \, i_{\scriptscriptstyle n}} 
x_{\scriptscriptstyle i_{\scriptscriptstyle 1}} \! \cdots \, x_{\scriptscriptstyle i_{\scriptscriptstyle n}} 
\;\;\;\;\;\;
h_{\scriptscriptstyle n} \, = \sum_{i_{\scriptscriptstyle 1} \le \, \cdots \, \le \, i_{\scriptscriptstyle n}} 
x_{\scriptscriptstyle i_{\scriptscriptstyle 1}} \! \cdots \, x_{\scriptscriptstyle i_{\scriptscriptstyle n}} 
\;\;\mathrm{and}\;\;\;
p_{\scriptscriptstyle n} = \, \sum_{i \ge 1} x_{\scriptscriptstyle i_{}}^{\scriptstyle n} 
\end{equation*} 
called the elementary symmetric functions, the complete symmetric functions and the power sums, respectively. 
The ring $\Lambda$ is also the completion of the graded polynomial ring $\mathbb{Z}[e_{\scriptscriptstyle 1}, e_{\scriptscriptstyle 2}, \ldots ],$ 
and thus, by the well-known identities among the standard symmetric functions, we have  
\begin{equation*}
\Lambda = \mathbb{Z}[\!\![e_{\scriptscriptstyle 1}, e_{\scriptscriptstyle 2}, \ldots ]\!\!]
= \mathbb{Z}[\!\![h_{\scriptscriptstyle 1}, h_{\scriptscriptstyle 2}, \ldots ]\!\!]
\;\;\mathrm{and}\;\;
\Lambda \otimes_{\scriptscriptstyle \mathbb{Z}} \mathbb{Q}= \mathbb{Q}[\!\![p_{\scriptscriptstyle 1}, p_{\scriptscriptstyle 2}, \ldots ]\!\!].
\end{equation*} 
In addition, for 
$
\lambda =
(\lambda_{\scriptscriptstyle 1} \ge \cdots \ge \lambda_{\scriptscriptstyle n} > 0)
$ 
we have the Schur functions 
\begin{equation*}
s_{\scriptscriptstyle \lambda} = 
\mathrm{det}(h_{\scriptscriptstyle \lambda_{\scriptscriptstyle i} - \, i \, + \, j})_{\scriptscriptstyle 1\le i, j \le n} 
\end{equation*} 
which also generate $\Lambda.$

\vskip5pt
{\bf Plethysm.} One defines (see \cite{McD} and \cite{GK}) the associative operation $``\circ"$ on $\Lambda,$ 
called {\it plethysm} (or {\it composition}), characterized by the following conditions:
\begin{enumerate} 
\item 
$
(f_{1} + f_{2}) \circ g = f_{1} \circ g  + f_{2} \circ g
$
\item 
$
(f_{1}f_{2}) \circ g = (f_{1} \circ g)(f_{2} \circ g) 
$ 

\item If $f = f(p_{\scriptscriptstyle 1}, p_{\scriptscriptstyle 2}, \ldots),$ we have
$
p_{\scriptscriptstyle n} \circ f = f(p_{\scriptscriptstyle n}, p_{\scriptscriptstyle 2n}, \ldots).
$ 
\end{enumerate} 
In other words, $\Lambda$ is a $\lambda$-ring, see \cite{Kn} and \cite{Fu-La} for generalities on $\lambda$-rings; it is the 
complete filtered $\lambda$-ring obtained by completing the free $\lambda$-ring on one generator 
$
\mathbb{Z}[e_{\scriptscriptstyle 1}, e_{\scriptscriptstyle 2}, \ldots ].
$ 
In this language 
$
p_{\scriptscriptstyle n} \circ f = \psi_{\scriptscriptstyle n}(f), 
$ 
where $\psi_{\scriptscriptstyle n}$ denote the Adams operations.

Consider the ring $\Lambda(\!\!(\hbar)\!\!)$ of Laurent series with the topology induced by the descending filtration
\begin{equation*}
F^{n}\Lambda(\!\!(\hbar)\!\!) = \Big\{\sum_{i} f_{i}\hbar^{i} : f_{i} \in F^{n - 2i}\Lambda \Big\}.
\end{equation*}
The plethysm extends to $\Lambda \times \Lambda(\!\!(\hbar)\!\!)$ by keeping conditions 1., 2. above, 
and replacing the last condition by: 
\begin{equation*}
p_{\scriptscriptstyle n} \circ f(\hbar, p_{\scriptscriptstyle 1}, p_{\scriptscriptstyle 2}, \ldots) 
= f(\hbar^{n}\!, p_{\scriptscriptstyle n}, p_{\scriptscriptstyle 2n}, \ldots).
\end{equation*} 
For $f \in F^{1}\Lambda(\!\!(\hbar)\!\!),$ put 
$
\Psi(f) : = \sum_{n \, \ge \, 1} \frac{1}{n} \psi_{\scriptscriptstyle n}(f),
$ 
where $\psi_{\scriptscriptstyle n}(f) = p_{\scriptscriptstyle n} \circ f,$ 
and let $\mathrm{Exp}(f)  = \exp(\Psi(f));$ we have  
\begin{equation*}
\mathrm{Exp}(f + g) = \mathrm{Exp}(f)\mathrm{Exp}(g).
\end{equation*} 
The mapping 
$
\Psi : F^{1}\Lambda(\!\!(\hbar)\!\!) \to F^{1}\Lambda(\!\!(\hbar)\!\!)
$ 
has an inverse 
$
\Psi^{-1} : F^{1}\Lambda(\!\!(\hbar)\!\!) \to F^{1}\Lambda(\!\!(\hbar)\!\!)
$ 
given by 
\begin{equation*}
\Psi^{-1}(f) = \sum_{m \, \ge \, 1} \frac{\mu(m)}{m} \psi_{\scriptscriptstyle m}(f)
\end{equation*} 
where $\mu(m)$ is the usual M\"obius function; see \cite[Lemma~20]{Moz}. Thus if we define 
$\mathrm{Log} : 1 + F^{1}\Lambda(\!\!(\hbar)\!\!) \to F^{1}\Lambda(\!\!(\hbar)\!\!)$ by 
$\mathrm{Log}(g) = \Psi^{-1}(\log g),$ then $\mathrm{Exp}$ and $\mathrm{Log}$ are inverses of each other.

Finally, one defines (\cite[Section~8]{GK}) an analogue of the Laplacian on $\Lambda(\!\!(\hbar)\!\!)$ by  
\begin{equation*}
\Delta = \sum_{n = 1}^{\infty} \hbar^{n} \left(\frac{n}{2}\frac{\partial^{2}}{\partial p_{\scriptscriptstyle n}^{2}} + 
\frac{\partial^{}}{\partial p_{\scriptscriptstyle 2 n}^{}}\right).
\end{equation*} 
Note that $\Delta$ preserves the filtration of $\Lambda(\!\!(\hbar)\!\!).$

\vskip5pt
{\bf The moduli spaces $\mathscr{M}_{g, n}$ and $\overline{\mathscr{M}}_{g, n}.$} For $g, n \in \mathbb{N}$ 
with $2(g - 1) + n > 0,$ let $\overline{\mathscr{M}}_{g, n}$ denote the proper and smooth Deligne-Mumford stack 
of \emph{stable} curves of \emph{arithmetic} genus $g$ with $n$ ordered distinct smooth points. Let 
$\mathscr{M}_{g, n} \subset \overline{\mathscr{M}}_{g, n}$ be the open substack of irreducible and 
non-singular curves; both $\mathscr{M}_{g, n}$ and $\overline{\mathscr{M}}_{g, n}$ are defined over 
$\mathrm{Spec}(\mathbb{Z}),$ and the group $\mathbb{S}_{n}$ acts on them by permuting the marked 
points on the curves. Moreover, we know that the boundary $\overline{\mathscr{M}}_{g, n} \setminus \mathscr{M}_{g, n}$ 
is a normal crossings divisor.

The stack $\overline{\mathscr{M}}_{g, n}$ admits a stratification determined by stable graphs of genus $g$ 
with $n$ ordered legs; see \cite[Chap.~XII.10]{ACG}. To each stable graph $G,$ there corresponds 
the smooth, locally closed stratum $\mathscr{M}(G) \subset \overline{\mathscr{M}}_{g, n}$ 
parametrizing curves with dual graph isomorphic to $G.$ The stratum $\mathscr{M}(G)$ is canonically isomorphic to 
the quotient stack
\small{\[ 
\bigg[\bigg(\prod_{v \in \mathrm{Vert}(G)} \mathscr{M}_{g(v),\, n(v)}\bigg) \slash \mathrm{Aut}(G)\bigg].
\]} 
Here $\mathrm{Vert}(G)$ denotes the set of vertices of $G,$ $g(v)$ is the geometric genus of the component 
(of a stable curve) corresponding to $v$ and $n(v)$ is the valence of the vertex $v.$ The automorphism group 
$\mathrm{Aut}(G)$ of $G$ is the set of graph automorphisms preserving the genus function $g$ and the 
ordering of the legs.

For generalities on these moduli spaces, see \cite{DM}, \cite{Knud}, \cite{ACG} 
and \cite{HM}.

\vskip5pt
{\bf Euler characteristics.} Put 
$
G_{\scriptscriptstyle \mathbb{Q}} = \mathrm{Gal}(\overline{\mathbb{Q}}\slash \mathbb{Q}),
$
and for a prime $\ell,$ denote by 
$
\mathbf{Rep}_{\mathbb{Q}_{\scriptscriptstyle \ell}}\!(G_{\scriptscriptstyle \mathbb{Q}})
$ 
the abelian category of $\ell$-adic Galois repre-\\sentations of $G_{\scriptscriptstyle \mathbb{Q}}.$ Let 
$
\mathrm{K}_{0}(\mathbf{Rep}_{\mathbb{Q}_{\scriptscriptstyle \ell}}\!(G_{\scriptscriptstyle \mathbb{Q}}))
$ 
denote the Grothendieck ring of this category; it carries a natural $\lambda$-ring structure, 
cf. \cite[expos\'e 5]{SGA6} or \cite{Getz0}, the $\lambda$-operations enjoying the property: 
\begin{equation*} 
\lambda^{\! m}([V]) = [\wedge^{\! m} V] \;\, \text{for\, $m \ge 0$.}
\end{equation*} 
We also recall that a semi-simple $\ell$-adic Galois representation $V$ is determined by the traces 
$\mathrm{Tr}(\sigma_{\! \scriptscriptstyle p} \, \vert \, V)$ of Frobenius elements $\sigma_{\! \scriptscriptstyle p}$ 
on the primes $p$ at which $V$ is unramified (see, for instance, \cite[Proposition~2.6]{DDT}).

Let $\mathscr{M}_{g, n \scriptscriptstyle \slash \overline{\mathbb{Q}}}$ denote the $\overline{\mathbb{Q}}$-stack corresponding to $\mathscr{M}_{g, n},$ i.e., the generic fiber $\mathscr{M}_{g, n \scriptscriptstyle \slash \mathbb{Q}}$ 
of $\mathscr{M}_{g, n} \to \mathrm{Spec}(\mathbb{Z})$ base changed from $\mathbb{Q}$ to $\overline{\mathbb{Q}}.$ 
The action of the symmetric group $\mathbb{S}_{n}$ on $\mathscr{M}_{g, n}$ induces an isotypic decomposition of 
the $\ell$-adic cohomology 
$
H^{\scriptscriptstyle i}_{\scriptscriptstyle c}
(\mathscr{M}_{g, n \scriptscriptstyle \slash \overline{\mathbb{Q}}}, \mathbb{Q}_{\scriptscriptstyle \ell})
$ 
as an $\mathbb{S}_{n}$-module, 
\begin{equation*}
H^{\scriptscriptstyle i}_{\scriptscriptstyle c}
(\mathscr{M}_{g, n \scriptscriptstyle \slash \overline{\mathbb{Q}}}, \mathbb{Q}_{\scriptscriptstyle \ell}) \cong 
\bigoplus_{\lambda \, \vdash \, n} 
H^{\scriptscriptstyle i}_{\scriptscriptstyle c, \lambda}
(\mathscr{M}_{g, n \scriptscriptstyle \slash \overline{\mathbb{Q}}}, \mathbb{Q}_{\scriptscriptstyle \ell})
\end{equation*} 
where for an irreducible representation $V_{\scriptscriptstyle \lambda}$ of $\mathbb{S}_{n}$ indexed by the partition 
$\lambda$ of $n,$  
\begin{equation*}
H^{\scriptscriptstyle i}_{\scriptscriptstyle c, \lambda}
(\mathscr{M}_{g, n \scriptscriptstyle \slash \overline{\mathbb{Q}}}, \mathbb{Q}_{\scriptscriptstyle \ell}) 
= V_{\scriptscriptstyle \lambda} \otimes \mathrm{Hom}_{\scriptscriptstyle \mathbb{S}_{n}}
(V_{\scriptscriptstyle \lambda}, H^{\scriptscriptstyle i}_{\scriptscriptstyle c}
(\mathscr{M}_{g, n \scriptscriptstyle \slash \overline{\mathbb{Q}}}, \mathbb{Q}_{\scriptscriptstyle \ell})).
\end{equation*} 
For a partition $\lambda$ of $n,$ put 
\begin{equation*}
\mathbf{e}_{\scriptscriptstyle c, \lambda}(\mathscr{M}_{g, n \scriptscriptstyle \slash \overline{\mathbb{Q}}}) = 
\sum_{i}\,  (- 1)^{i}[H^{\scriptscriptstyle i}_{\scriptscriptstyle c, \lambda}
(\mathscr{M}_{g, n \scriptscriptstyle \slash \overline{\mathbb{Q}}}, \mathbb{Q}_{\scriptscriptstyle \ell})]
\in \mathrm{K}_{0}(\mathbf{Rep}_{\mathbb{Q}_{\scriptscriptstyle \ell}}\!(G_{\scriptscriptstyle \mathbb{Q}}));
\end{equation*} 
in addition, we fix throughout a finite field $\mathbb{F}$ of characteristic different from $\ell$ and an algebraic closure 
of it $\overline{\mathbb{F}},$ and define similarly the Euler characteristic 
$
\mathbf{e}_{\scriptscriptstyle c, \lambda}(\mathscr{M}_{g, n \scriptscriptstyle \slash \overline{\mathbb{F}}}). 
$ 
Let $\mathcal{V}(\!\!(g, n)\!\!) = H^{*}_{\scriptscriptstyle c}
(\mathscr{M}_{g, n \scriptscriptstyle \slash \overline{\mathbb{Q}}}, \mathbb{Q}_{\scriptscriptstyle \ell}),$ 
and set 
\begin{equation*} 
\mathrm{ch}_{\scriptscriptstyle n}(\mathcal{V}(\!\!(g, n)\!\!)\!) : =
\sum_{\lambda \, \vdash \, n} \frac{1}{\dim \, V_{\scriptscriptstyle \lambda}} 
\, \mathbf{e}_{\scriptscriptstyle c, \lambda}
(\mathscr{M}_{g, n \scriptscriptstyle \slash \overline{\mathbb{Q}}}) s_{\scriptscriptstyle \lambda}
\end{equation*} 
see also \cite{BeTo} where this characteristic is computed for small values of $g$ and $n.$

Letting 
$ 
F : \mathscr{M}_{g, n \scriptscriptstyle \slash \overline{\mathbb{F}}}\to 
\mathscr{M}_{g, n \scriptscriptstyle \slash \overline{\mathbb{F}}}
$ 
denote the Frobenius morphism, we note that by Grothendieck's fixed point formula \cite{Beh1, Beh2} 
\!\footnote{The smooth Deligne-Mumford stack $\mathscr{M}_{g, n}$ is of finite type and has 
relative dimension $3(g - 1) + n$ over $\mathrm{Spec}(\mathbb{Z}),$ thus fulfilling the conditions 
of \cite[Theorem~3.1.2]{Beh1}.}, the trace of the geometric Frobenius $F^{*}$ on the characteristic 
$\mathrm{ch}_{\scriptscriptstyle n}$ of the graded $\mathbb{S}_{n}$-module 
$
H^{*}_{\scriptscriptstyle c}
(\mathscr{M}_{g, n \scriptscriptstyle \slash \overline{\mathbb{F}}}, \mathbb{Q}_{\scriptscriptstyle \ell})
$ 
is given by 
\begin{equation*} 
t_{g, n}({\bf p}) = t_{g, n}(p_{\scriptscriptstyle 1}, p_{\scriptscriptstyle 2}, \ldots) : \, = 
\sum_{\lambda \, \vdash \, n} \frac{1}{\dim \, V_{\scriptscriptstyle \lambda}} 
\mathrm{Tr}(F^{*} \vert \, 
\mathbf{e}_{\scriptscriptstyle c, \lambda}(\mathscr{M}_{g, n \scriptscriptstyle \slash \overline{\mathbb{F}}})) 
s_{\scriptscriptstyle \lambda} = \frac{1}{n!}\sum_{\sigma \, \in \, \mathbb{S}_{\scriptscriptstyle n}} 
\left| \mathscr{M}_{g, n \scriptscriptstyle \slash \overline{\mathbb{F}}}^{\sigma F} \right | p_{\mathrm{c}(\sigma)}
\end{equation*} 
where $\mathrm{c}(\sigma)$ denotes the cycle type of $\sigma,$ and if 
$
\mathrm{c}(\sigma) = \left(1^{\rho(1)}\!, \ldots, n^{\rho(n)} \right)
$ 
then 
$
p_{\mathrm{c}(\sigma)} = p_{\scriptscriptstyle 1}^{\rho(1)} \cdots \, p_{\scriptscriptstyle n}^{\rho(n)}. 
$ 
Here we have also used Frobenius' formula 
\begin{equation*} 
s_{\scriptscriptstyle \lambda} = \frac{1}{n!}\sum_{\sigma \, \in \, \mathbb{S}_{\scriptscriptstyle n}} 
\chi^{\lambda}(\sigma) p_{\mathrm{c}(\sigma)}
\end{equation*} 
$\chi^{\lambda}$ being the character of $V_{\scriptscriptstyle \lambda}.$ Since 
$\left|\mathscr{M}_{g, n \scriptscriptstyle \slash \overline{\mathbb{F}}}^{\sigma F}\right|$ 
depends only upon the cycle type of $\sigma,$ we can also write 
\begin{equation*} 
t_{g, n}({\bf p}) \, = \sum_{\rho \, \vdash \, n} \, 
\big|\mathscr{M}_{g, n \scriptscriptstyle \slash \overline{\mathbb{F}}}^{\rho F}\big| 
\, \frac{p_{\scriptscriptstyle \rho}}{z_{\rho}}
\end{equation*} 
where, for convenience, we set 
$ 
\big|\mathscr{M}_{g, n \scriptscriptstyle \slash \overline{\mathbb{F}}}^{\rho F} \big| : = 
\left|\mathscr{M}_{g, n \scriptscriptstyle \slash \overline{\mathbb{F}}}^{\sigma F}\right |
$ 
for $\sigma \in \mathbb{S}_{\scriptscriptstyle n}$ with $\mathrm{c}(\sigma) = \rho,$ and 
\begin{equation*}
z_{\rho} = \prod_{i = 1}^{n} i^{\rho(i)} \rho(i)! \;\;\; 
\text{if $\rho = \big(1^{\rho(1)}\!, \ldots, n^{\rho(n)} \big)$.}
\end{equation*} 
Note that, for $k \ge 1,$ we have 
\begin{equation*} 
\psi_{\scriptscriptstyle k}^{}(t_{g, n})({\bf p}) \, = 
\sum_{\rho \, \vdash \, n} \, \big|\mathscr{M}_{g, n \scriptscriptstyle \slash \overline{\mathbb{F}}}^{\rho F^{k}} \big| 
\, \frac{p_{\scriptscriptstyle k \rho}}{z_{\rho}}.
\end{equation*} 
Now the characteristic of the stable $\mathbb{S}$-module $\mathcal{V} = \{\mathcal{V}(\!\!(g, n)\!\!)\!\}$ 
(that is, the representation $\mathcal{V}$ of the groupoid $\mathbb{S} \, = \coprod_{n\ge 0}\mathbb{S}_{n}$) 
is defined by the formal Laurent series 
\begin{equation*} 
\mathbb{C}\mathrm{h}(\mathcal{V}) \; = \sum_{2(g - 1) + n \, > \, 0}\hbar^{g - 1}
\mathrm{ch}_{\scriptscriptstyle n}(\mathcal{V}(\!\!(g, n)\!\!)\!)
\end{equation*} 
with the coefficients 
\begin{equation*} 
\sum_{n\, \ge \, \textrm{max}\{0,\, 3 - 2g\}}
\mathrm{ch}_{\scriptscriptstyle n}(\mathcal{V}(\!\!(g, n)\!\!)\!) \in 
\mathrm{K}_{0}(\mathbf{Rep}_{\mathbb{Q}_{\scriptscriptstyle \ell}}\!(G_{\scriptscriptstyle \mathbb{Q}}))
[\!\![h_{\scriptscriptstyle 1}, h_{\scriptscriptstyle 2}, \ldots ]\!\!].
\end{equation*} 
The corresponding generating series of $t_{g, n}({\bf p})$ will be denoted by $\mathrm{T}(\mathcal{V}).$

The free modular operad $\mathbb{M}\mathcal{V} = \{\mathbb{M}\mathcal{V}(\!\!(g, n)\!\!)\!\}$ 
generated by $\mathcal{V}$ is obtained by taking 
$
\mathbb{M}\mathcal{V}(\!\!(g, n)\!\!) = H^{*}_{\scriptscriptstyle c}
(\overline{\mathscr{M}}_{g, n \scriptscriptstyle \slash \overline{\mathbb{Q}}}, \mathbb{Q}_{\scriptscriptstyle \ell});
$ 
the characteristic $\mathbb{C}\mathrm{h}(\mathbb{M}\mathcal{V})$ and the corresponding generating 
series $\mathrm{T}(\mathbb{M}\mathcal{V})$ are defined as for $\mathcal{V}.$

The connection between the characteristics $\mathbb{C}\mathrm{h}(\mathcal{V})$ and 
$\mathbb{C}\mathrm{h}(\mathbb{M}\mathcal{V})$ is given by the following theorem of Getzler and Kapranov \cite[Theorem~8.13]{GK}:

\vskip10pt
\begin{thm}\label{Theorem GK1} --- We have 
\begin{equation*}
\mathbb{C}\mathrm{h}(\mathbb{M}\mathcal{V}) = 
\mathrm{Log}(\exp(\Delta)\mathrm{Exp}(\mathbb{C}\mathrm{h}(\mathcal{V}))).
\end{equation*}
\end{thm} 
Here $\mathrm{Exp}(-)$ and $\mathrm{Log}(-)$ are defined as before.

\vskip5pt
{\bf Integral representation.} Following Getzler and Kapranov, we shall now express 
$\mathbb{C}\mathrm{h}(\mathbb{M}\mathcal{V})$ as a formal Fourier transform. 
\!The resul-\\ting formula is in complete analogy with the formula \eqref{eq: Wick} in Wick's theorem. 

For a partition $\rho = \big(1^{\rho(1)}\!, 2^{\rho(2)}\!, \ldots \big),$ where $\rho(j) = 0$ for all but finitely many $j,$ put, 
as before, 
$
p_{\scriptscriptstyle \rho} = p_{\scriptscriptstyle 1}^{\rho(1)} p_{\scriptscriptstyle 2}^{\rho(2)}\cdots. 
$ 
Let $\Lambda_{\mathrm{alg}}$ denote the space of {\it finite} linear combinations of the $p_{\scriptscriptstyle \rho}.$ 
On $\mathrm{Spec}(\Lambda_{\mathrm{alg}} \otimes \mathbb{R}) \cong \mathbb{R}^{\infty}$ with 
coordinates $p_{\scriptscriptstyle 1}, p_{\scriptscriptstyle 2}, \ldots,$ let 
$d\mu$ denote the formal Gaussian measure defined by 
\begin{equation*} 
\begin{split} 
d\mu  \; & = \prod_{n - \mathrm{odd}} 
e^{- p_{\scriptscriptstyle n}^{2}\slash 2 n \hbar^{\! n}}
\frac{dp_{\scriptscriptstyle n}}{\sqrt{2 \pi n \hbar^{n}}}
\, \prod_{n - \mathrm{even}} e^{- p_{\scriptscriptstyle n}^{2}\slash 2 n \hbar^{\! n} + \, 
p_{\scriptscriptstyle n}^{}\slash n \hbar^{\! n\slash 2}} 
\frac{dp_{\scriptscriptstyle n}}{e^{1\slash 2 n}\, \sqrt{2 \pi n \hbar^{n}}}\\
& = \mathrm{Exp}(- e_{\scriptscriptstyle 2} \slash \hbar) \prod_{n = 1}^{\infty} \frac{dp_{\scriptscriptstyle n}}
{e^{\varepsilon_{\! \scriptscriptstyle n} \slash 2 n}\, \sqrt{2 \pi n \hbar^{n}}}
\end{split}
\end{equation*} 
where $\varepsilon_{\! \scriptscriptstyle n} = 0$ or $1$ according as $n$ is odd or even. With this measure, 
for each monomial 
$
p_{\scriptscriptstyle \rho} = p_{\scriptscriptstyle 1}^{\rho(1)} p_{\scriptscriptstyle 2}^{\rho(2)}\cdots, 
$ 
define the integral  
\begin{equation*} 
\begin{split} 
\int_{\mathbb{R}^{\infty}}^{\! *} \, p_{\scriptscriptstyle \rho} \, d\mu(\mathbf{p})
\; & = \prod_{n - \mathrm{odd}} \int_{- \infty}^{\infty} p_{\scriptscriptstyle n}^{\rho(n)} 
e^{- p_{\scriptscriptstyle n}^{2}\slash 2 n \hbar^{\! n}}\frac{dp_{\scriptscriptstyle n}}{\sqrt{2 \pi n \hbar^{n}}}\\
& \cdot \prod_{n - \mathrm{even}} \int_{- \infty}^{\infty} p_{\scriptscriptstyle n}^{\rho(n)} 
e^{- p_{\scriptscriptstyle n}^{2}\slash 2 n \hbar^{\! n} + \, 
p_{\scriptscriptstyle n}^{}\slash n \hbar^{\! n\slash 2}} 
\frac{dp_{\scriptscriptstyle n}}{e^{1\slash 2 n}\, \sqrt{2 \pi n \hbar^{n}}}.
\end{split}
\end{equation*} 
Notice that almost all integrals in the products equal $1.$ We extend 
$
\int_{\mathbb{R}^{\infty}}^{\! *}
$ 
to a map on 
$
\mathrm{K}_{0}(\mathbf{Rep}_{\mathbb{Q}_{\scriptscriptstyle \ell}}\!(G_{\scriptscriptstyle \mathbb{Q}}))
[\!\![h_{\scriptscriptstyle 1}, h_{\scriptscriptstyle 2}, \ldots ]\!\!](\!\!(\hbar)\!\!) 
$ 
by linearity.

With this definition, we have the following interpretation of the formula in Theorem \ref{Theorem GK1}, see 
\cite[Theorem~8.18]{GK}:

\vskip10pt
\begin{thm}\label{Theorem GK2} --- We have 
\begin{equation} \label{eq: Getz-Kapr-original-integral-represent} 
\hbar^{- 1}h_{\scriptscriptstyle 2} + 
\mathbb{C}\mathrm{h}(\mathbb{M}\mathcal{V}) = 
\mathrm{Log}\left(\int_{\mathbb{R}^{\infty}}^{\! *} \mathrm{Exp}(\hbar^{- 1} p_{\scriptscriptstyle 1} q_{\scriptscriptstyle 1} 
+ \mathbb{C}\mathrm{h}(\mathcal{V}))\, d\mu(\mathbf{p}) \right)
\end{equation} 
where the left-hand side is considered as a function of 
$
\mathbf{q} = (q_{\scriptscriptstyle 1}, q_{\scriptscriptstyle 2}, \ldots) 
$ 
and $\hbar.$
\end{thm}

For our purposes it will be more convenient to switch the roles of 
$
\mathbf{p} = (p_{\scriptscriptstyle 1}, p_{\scriptscriptstyle 2}, \ldots) 
$ 
and 
$
\mathbf{q} = (q_{\scriptscriptstyle 1}, q_{\scriptscriptstyle 2}, \ldots), 
$ 
and write \eqref{eq: Getz-Kapr-original-integral-represent} in the form 
\begin{equation} \label{eq: Getz-Kapr-start-point-integral-representation} 
\mathrm{Exp}(\mathbb{C}\mathrm{h}(\mathbb{M}\mathcal{V})) = 
\int_{\mathbb{R}^{\infty}} e^{\mathcal{K}(\mathbf{p}, \mathbf{q}, \hbar)} \, d^{*}\!\mathbf{q}
\end{equation} 
with $\mathcal{K}(\mathbf{p}, \mathbf{q}, \hbar)$ defined by 
\begin{equation*} \label{eq: function-kappa} 
\mathcal{K}(\mathbf{p}, \mathbf{q}, \hbar) = 
- \sum_{m = 1}^{\infty} \left\{\left(q_{\scriptscriptstyle m} \! - p_{\scriptscriptstyle m} \! - \varepsilon_{\scriptscriptstyle m}\hbar^{\scriptscriptstyle m \slash 2} \right)^{2} \! \slash 2 m \hbar^{\scriptscriptstyle m} \! 
\right\} \, +\, \Psi(\mathbb{C}\mathrm{h}(\mathcal{V}))
\end{equation*} 
and the measure 
\begin{equation*} 
d^{*}\!{\mathbf{q}} \, = \prod_{m = 1}^{\infty} \frac{dq_{\scriptscriptstyle m}}{\sqrt{2 \pi m \hbar^{m}}}.
\end{equation*} 
Formula \eqref{eq: Getz-Kapr-start-point-integral-representation} follows easily by exponentiating 
\eqref{eq: Getz-Kapr-original-integral-represent} and applying the identities 
$e_{\scriptscriptstyle 2}^{} \! = (q_{\scriptscriptstyle 1}^{2} - q_{\scriptscriptstyle 2}^{}) \slash 2$ and 
$h_{\scriptscriptstyle 2}^{} \!= (p_{\scriptscriptstyle 1}^{2} + p_{\scriptscriptstyle 2}^{}) \slash 2.$

\section{A semi-classical expansion} 
To express the coefficients of $\mathbb{C}\mathrm{h}(\mathbb{M}\mathcal{V})$ in terms of those of 
$\mathbb{C}\mathrm{h}(\mathcal{V}),$ both $\mathbb{C}\mathrm{h}(\mathcal{V})$ and 
$\mathbb{C}\mathrm{h}(\mathbb{M}\mathcal{V})$ considered as formal Laurent series in $\hbar,$ we shall study 
the integral \eqref{eq: Getz-Kapr-start-point-integral-representation} corresponding to the generating series 
$\mathrm{T}(\mathcal{V})$ of $t_{g, n}({\bf p})$ over any finite field $\mathbb{F}$ of characteristic different from 
$\ell.$ Here, we recall that the Euler characteristics $\mathbf{e}_{\scriptscriptstyle c, \lambda}$ 
are elements of the Grothendieck ring 
$
\mathrm{K}_{0}(\mathbf{Rep}_{\mathbb{Q}_{\scriptscriptstyle \ell}}\!(G_{\scriptscriptstyle \mathbb{Q}})).
$

In what follows, we shall denote by $\mathrm{ch}_{\scriptscriptstyle g}(\mathcal{V})$ 
(resp. $\mathrm{ch}_{\scriptscriptstyle g}(\mathbb{M}\mathcal{V})$) the coefficient of $\hbar^{g - 1}$ 
in the series $\mathrm{T}(\mathcal{V})$ (resp. $\mathrm{T}(\mathbb{M}\mathcal{V})$), i.e., 
\begin{equation*} 
\mathrm{ch}_{\scriptscriptstyle g}(\mathcal{V})(\mathbf{q}) \; = 
\sum_{n\, \ge \, \textrm{max}\{0,\, 3 - 2g\}}\, \sum_{\rho \, \vdash \, n} \, 
\big|\mathscr{M}_{g, n \scriptscriptstyle \slash \overline{\mathbb{F}}}^{\rho F}\big| 
\, \frac{q_{\scriptscriptstyle \rho}}{z_{\rho}}
\end{equation*} 
and 
\begin{equation*}
\mathrm{ch}_{\scriptscriptstyle g}(\mathbb{M}\mathcal{V})(\mathbf{p}) \; = 
\sum_{n\, \ge \, \textrm{max}\{0,\, 3 - 2g\}}\, \sum_{\rho \, \vdash \, n} \, 
\big|\overline{\mathscr{M}}_{g, n \scriptscriptstyle \slash \overline{\mathbb{F}}}^{\rho F}\big| 
\, \frac{p_{\scriptscriptstyle \rho}}{z_{\rho}}.
\end{equation*} 
To make use of Theorem \ref{Theorem GK2} of Getzler and Kapranov in the form 
\eqref{eq: Getz-Kapr-start-point-integral-representation} to obtain formulas expressing 
the coefficients $\mathrm{ch}_{\scriptscriptstyle g}(\mathbb{M}\mathcal{V})$ in terms of the coefficients 
$\mathrm{ch}_{\scriptscriptstyle g'}(\mathcal{V})$ with $g' \le g,$ we shall perform a semi-classical expansion 
of the corresponding integral giving $T(\mathbb{M}\mathcal{V}).$ The analogous expansion in the context of the usual 
orbifold Euler characteristics $\chi(\mathscr{M}_{g, n})$ and $\chi(\overline{\mathscr{M}}_{g, n})$ is discussed 
by Bini and Harer in \cite[3.1]{BH}. Our arguments are a natural extension of theirs.

\subsection{The critical points} 
To apply the principle of stationary phase to the integral giving $T(\mathbb{M}\mathcal{V}),$ 
we first need to determine the critical points of the function (still denoted by $\mathcal{K}(\mathbf{p}, \mathbf{q}, \hbar)$) 
given by 
\begin{equation*} 
\mathcal{K}(\mathbf{p}, \mathbf{q}, \hbar) = 
- \sum_{m = 1}^{\infty} \left\{\left(q_{\scriptscriptstyle m} \! - p_{\scriptscriptstyle m} \! - \varepsilon_{\scriptscriptstyle m}\hbar^{\scriptscriptstyle m \slash 2} \right)^{2} \! \slash 2 m \hbar^{\scriptscriptstyle m} \! 
\right\} \, +\, \Psi(T(\mathcal{V})).
\end{equation*} 
More precisely, for 
$m = 1, 2, \ldots,$ we have to find the power series 
\begin{equation*} 
\bar{q}_{\scriptscriptstyle m}^{} \! = \bar{q}_{\scriptscriptstyle m}^{}\!(\mathbf{p}, \hbar) 
= \sum_{s = 0}^{\infty} \bar{c}_{m, s}^{}(\mathbf{p})\hbar^{s}
\end{equation*} 
satisfying the system of differential equations 
\begin{equation*} 
\frac{\partial \mathcal{K}}{\partial q_{\scriptscriptstyle m}^{}}(\mathbf{p}, \bar{\mathbf{q}}, \hbar) = 0
\qquad
\text{($\bar{\mathbf{q}} = (\bar{q}_{\scriptscriptstyle 1}^{}, \, \bar{q}_{\scriptscriptstyle 2}^{}, \ldots)$)}
\end{equation*} 
for $m = 1, 2, \ldots,$ or written explicitly, 
\begin{equation} \label{eq: barq-version1}
\bar{q}_{\scriptscriptstyle m}^{} \! = p_{\scriptscriptstyle m}^{}  
+  \varepsilon_{\scriptscriptstyle m}^{}\hbar^{m \slash 2}  +\, 
\sum_{k \mid m} \, \frac{m}{k} \, \sum_{g = 0}^{\infty}  \hbar^{kg + m - k} \,
\frac{\partial \psi_{\scriptscriptstyle k}^{}(\mathrm{ch}_{\scriptscriptstyle g}(\mathcal{V}))}
{\partial q_{\scriptscriptstyle m}^{}}(\bar{\mathbf{q}}).
\end{equation} 
Here $\psi_{\scriptscriptstyle k}^{}(\mathrm{ch}_{\scriptscriptstyle g}(\mathcal{V}))$ is given by 
\begin{equation*} 
\psi_{\scriptscriptstyle k}^{}(\mathrm{ch}_{\scriptscriptstyle g}(\mathcal{V}))(\mathbf{q}) \; = 
\sum_{n\, \ge \, \textrm{max}\{0,\, 3 - 2g \}}\, \sum_{\rho \, \vdash \, n} \, 
\big|\mathscr{M}_{g, n \scriptscriptstyle \slash \overline{\mathbb{F}}}^{\rho F^{k}}\big|\, 
\frac{q_{\scriptscriptstyle k \rho}}{z_{\rho}}. 
\end{equation*} 
For $m \ge 1,$ define $c_{\scriptscriptstyle m}^{}(\mathbf{q})$ by
\begin{equation*}
c_{\scriptscriptstyle m}^{}(\mathbf{q}) =  
q_{\scriptscriptstyle m}^{}  -  
\frac{\partial \psi_{\scriptscriptstyle m}^{}(\mathrm{ch}_{\scriptscriptstyle 0}(\mathcal{V}))}
{\partial q_{\scriptscriptstyle m}^{}}(\mathbf{q}).
\end{equation*} 
Note that $c_{\scriptscriptstyle m}^{} \! = \psi_{\scriptscriptstyle m}^{}(c_{\scriptscriptstyle 1}^{}).$ If we put 
$C(\mathbf{q}) = (c_{\scriptscriptstyle m}^{}(\mathbf{q}))_{\scriptscriptstyle m \ge 1}$ and 
$\bar{C}(\mathbf{p}) = (\bar{c}_{m, {\scriptscriptstyle 0}}^{}(\mathbf{p}))_{\scriptscriptstyle m \ge 1},$ 
then equating the constant terms (i.e., setting $\hbar = 0$) in both sides of \eqref{eq: barq-version1}, 
we must have
\begin{equation*}
C(\bar{C}(\mathbf{p})) = \mathbf{p}.
\end{equation*}

\vskip10pt
\begin{rem} \!\!\!In \cite[Lemma (2.8)]{KiLe}, Kisin and Lehrer obtained the formula
\begin{equation*}  
\frac{\big|\mathscr{M}_{{\scriptscriptstyle 0}, n \scriptscriptstyle \slash \overline{\mathbb{F}}}^{\rho F}\big|}
{z_{\rho}} = \frac{\prod_{i \ge 1} \binom{s_{i}(q) \slash i}{\rho(i)}}{q(q^{2} - 1)} \;\; \qquad \;\;  
\text{\big(with $\rho = \big(1^{\rho(1)}\!, \ldots,\, n^{\rho(n)}\big)$ and $|\mathbb{F}| = q$ \!\big)}
\end{equation*} 
where $s_{\scriptscriptstyle 1}^{}\!(q) = q + 1,$ and 
$ 
s_{\scriptscriptstyle i}^{}(q) = \sum_{d \mid i} \, \mu\left(\frac{i}{d}\right)q^{d} 
$ 
for $i \ge 2.$ Accordingly, we can write 
\begin{equation*} 
\mathrm{ch}_{\scriptscriptstyle 0}^{}(\mathcal{V})(\mathbf{p})  
= \frac{\left\{\prod_{n \ge 1}(1 + p_{\scriptscriptstyle n}^{})^{s_{\scriptscriptstyle n}^{}\!(q)\slash n} \right \}  -  1}{q(q^{2} - 1)} 
- \frac{p_{\scriptscriptstyle 1}^{2}}{2(q - 1)} - \frac{p_{\scriptscriptstyle 1}^{}}{q(q - 1)} 
- \frac{p_{\scriptscriptstyle 2}^{}}{2(q + 1)};
\end{equation*} 
see also \cite{Getz1}. Thus we can express  
\begin{equation} \label{eq: explicit c1}
c_{\scriptscriptstyle 1}^{}(\mathbf{p})  = p_{\scriptscriptstyle 1}^{}  
- \frac{\partial \mathrm{ch}_{\scriptscriptstyle 0}^{}(\mathcal{V})}{\partial p_{\scriptscriptstyle 1}^{}}(\mathbf{p}) 
= \frac{q \, p_{\scriptscriptstyle 1}^{}}{q - 1} \, - \, 
\frac{\left\{\prod_{n \ge 1}(1 + p_{\scriptscriptstyle n}^{})^{\!\frac{1}{n}\!\sum_{d \mid n} 
\mu\left(\! \frac{n}{d} \!\right)q^{d}}\right\} - 1}{q(q - 1)}.
\end{equation} 
\end{rem} 
Since $c_{\scriptscriptstyle 1}^{}(\mathbf{q})$ has no constant term, and the linear part of $C(\mathbf{q})$ 
is easily seen to be invertible, it follows that $C(\mathbf{q})$ admits, indeed, a compositional inverse. Note that 
\begin{equation*}
\psi_{\scriptscriptstyle m}^{}(\bar{c}_{\scriptscriptstyle 1, 0}^{})(C(\mathbf{p})) = p_{\scriptscriptstyle m}^{} \circ \,
\bar{c}_{\scriptscriptstyle 1, 0}^{}(C(\mathbf{p})) = p_{\scriptscriptstyle m}^{} 
\qquad  \text{(for all $m \ge 1$)}
\end{equation*} 
and thus $\bar{c}_{m, {\scriptscriptstyle 0}}^{} = \psi_{\scriptscriptstyle m}^{}(\bar{c}_{\scriptscriptstyle 1, 0}^{}),$ 
that is, $c_{\scriptscriptstyle 1}^{}$ and $\bar{c}_{\scriptscriptstyle 1, 0}^{}$ are {\it plethystic} inverses.

\vskip10pt
\begin{prop}\label{unique-formal-sol-barq} --- The system \eqref{eq: barq-version1} has a unique solution of the form 
\begin{equation*} 
\bar{q}_{\scriptscriptstyle m}^{}\!(\mathbf{p}, \hbar) = \bar{c}_{m, {\scriptscriptstyle 0}}^{}(\mathbf{p})  
\, +  \, O\big(\hbar^{(m + 1 - \varepsilon_{m}) \slash 2} \big) 
\end{equation*} 
for all $m\ge 1,$ with $\varepsilon_{m}^{} \! = 0$ or $1$ according as $m$ is odd or even. The coefficients 
$\bar{c}_{m,  m \slash 2}^{}$ of this solution are given by the formula
\begin{equation*}
\bar{c}_{m, m \slash 2}^{} = \frac{1  +  2 \psi_{\scriptscriptstyle \frac{m}{2}}^{}
\!\left(\frac{\partial \mathrm{ch}_{\scriptscriptstyle 0}^{}(\mathcal{V})}{\partial p_{\scriptscriptstyle 2}^{}} \!\right)}
{1  -  \psi_{\scriptscriptstyle m}^{}\!\left(\frac{\partial^{2}  \mathrm{ch}_{\scriptscriptstyle 0}^{}(\mathcal{V})}
{\partial p_{\scriptscriptstyle 1}^{2}} \!\right)}\, \circ \, \bar{c}_{\scriptscriptstyle 1, 0}^{}
\end{equation*} 
for all even $m \ge 2.$
\end{prop}

\begin{proof} \!\!\!\!Splitting the right-hand side of \eqref{eq: barq-version1} according as $kg + m - k$ 
(the exponent of $\hbar$) is zero (i.e., $g = 0$ and $k = m$) or not, we rewrite \eqref{eq: barq-version1} as \begin{equation} \label{barq-version2}
c_{m}^{}(\bar{{\bf q}}) \, - \, p_{m}^{} =  \varepsilon_{m}^{}\hbar^{m \slash 2} \, + \, 
\sum\,  \bigg(\! \frac{m\, \hbar^{k(g - 1) + m}}{k} \!\bigg)
\frac{\partial \psi_{\scriptscriptstyle k}^{}(\mathrm{ch}_{\scriptscriptstyle g}(\mathcal{V}))}
{\partial q_{\scriptscriptstyle m}^{}}(\bar{\mathbf{q}}) 
\end{equation} 
the sum being over $k\mid m$ and $g\ge 0$ such that $k(g - 1) + m > 0.$ Equating the coefficients of $\hbar^{s}\!,$ 
for any $s \ge 1,$ on both sides of \eqref{barq-version2}, we get recursive relations of the form 
\begin{equation} \label{eq: system-determine-qn}
\sum_{r \ge 1} \bar{c}_{m r, s}^{}(\mathbf{p}) 
\psi_{\scriptscriptstyle m}^{}\!\left(\frac{\partial c_{\scriptscriptstyle 1}^{}}
{\partial p_{\scriptscriptstyle r}^{}}\right)\!(\bar{C}(\mathbf{p})) = R_{m}(\mathbf{p}) 
\qquad \text{(for all $m \ge 1$)}
\end{equation} 
where $R_{m}(\mathbf{p})$ is an expression involving only coefficients $\bar{c}_{t,  s'}(\mathbf{p})$ ($t \ge 1$) with 
$s' < s.$ Notice that the right-hand side of \eqref{barq-version2} is $O\big(\hbar^{(m + 1 - \varepsilon_{m}) \slash 2} \big)$ \!\footnote{The exponent $k(g - 1) + m$ of $\hbar$ is at least $(m + 1 - \varepsilon_{m}) \slash 2.$}; thus the coefficient of 
$\hbar^{s}$ in $c_{m}^{}(\bar{\mathbf{q}})$ must be zero when $m - \varepsilon_{m}^{} \ge 2s,$ that is, $m \ge 2s + 1.$

We proceed now by induction on $s.$ Having the coefficients 
$
\bar{c}_{m, \scriptscriptstyle 0}^{} = \psi_{\scriptscriptstyle m}^{}(\bar{c}_{\scriptscriptstyle 1, 0}^{})
$ 
already determined, assume that $\bar{c}_{m, s'}^{}(\mathbf{p}),$ for all $s' \!< s$ and $m \ge 1,$ were also 
determined, and that, for every $s' \! < s,$ we have $\bar{c}_{m, s'}^{} \! = 0$ for all $m \ge 2s' + 1.$ Under these assumptions, we have $R_{m}(\mathbf{p}) = 0$ if $m \ge 2s + 1.$ Indeed, the right-hand side of 
\eqref{barq-version2} is certainly $O(\hbar^{s + 1}),$ and the coefficients in 
$c_{\scriptscriptstyle m}^{}(\bar{\mathbf{q}})$ getting into $R_{m}(\mathbf{p})$ are of the form 
$\bar{c}_{mr, s'}^{}$ with $s'  \! < s;$ these are all zero, since $mr  > 2s' + 1.$ 
Let $\bar{c}_{m,  s}^{} : = 0$ for $m \ge 2s + 1,$ hence 
\eqref{eq: system-determine-qn} holds trivially for these values of $m.$ 
From \eqref{eq: explicit c1}, it is easy to see that 
$ 
\psi_{\scriptscriptstyle m}^{}\Big(\! \frac{\partial c_{\scriptscriptstyle 1}^{}}
{\partial p_{\scriptscriptstyle 1}^{}} \!\Big)(\bar{C}(\mathbf{p}))
$ 
is non-zero for all $m \ge 1,$ implying that the linear system
\begin{equation*}
\sum_{1 \le  r  \le \frac{2s}{m}} \bar{c}_{m r,  s}^{}(\mathbf{p}) 
\psi_{\scriptscriptstyle m}^{}\!\left(\! \frac{\partial c_{\scriptscriptstyle 1}^{}}
{\partial p_{\scriptscriptstyle r}^{}} \!\right)\!(\bar{C}(\mathbf{p})) = R_{m}(\mathbf{p}) 
\qquad \text{($1\le m \le 2s$)}
\end{equation*} 
has a unique solution; thus \eqref{eq: system-determine-qn} has a unique solution of the form we asserted. 
This completes the induction.

To compute the coefficients $\bar{c}_{m, m \slash 2}^{}(\mathbf{p}),$ we just notice that 
\eqref{eq: system-determine-qn} corresponding to $s = m \slash 2$ is 
\begin{equation*} 
\bar{c}_{m, m \slash 2}^{}(\mathbf{p}) \, - \, \bar{c}_{m,  m \slash 2}^{}(\mathbf{p}) 
\psi_{\scriptscriptstyle m}^{}\!\left(\frac{\partial^{2} 
\mathrm{ch}_{\scriptscriptstyle 0}^{}(\mathcal{V})}{\partial p_{\scriptscriptstyle 1}^{2}} \!\right)(\bar{C}(\mathbf{p})) 
= 1 \, + \, 2 \psi_{\scriptscriptstyle  \frac{m}{2}}^{}\!\left(\frac{\partial  \mathrm{ch}_{\scriptscriptstyle 0}^{}(\mathcal{V})}
{\partial p_{\scriptscriptstyle 2}^{}} \!\right)(\bar{C}(\mathbf{p})).
\end{equation*} 
Since 
$
\bar{c}_{r,  {\scriptscriptstyle 0}}^{} = \psi_{\scriptscriptstyle r}^{}(\bar{c}_{\scriptscriptstyle 1,  0}^{}),
$ 
for all $r \ge 1,$ it follows that 
\begin{equation*}
\bar{c}_{m,  m \slash 2}^{} = \frac{1  +  2 \psi_{\scriptscriptstyle \frac{m}{2}}^{}
\!\left(\frac{\partial \mathrm{ch}_{\scriptscriptstyle 0}^{}(\mathcal{V})}{\partial p_{\scriptscriptstyle 2}^{}} \!\right)}
{1  -  \psi_{\scriptscriptstyle m}^{}\!\left(\frac{\partial^{2}  \mathrm{ch}_{\scriptscriptstyle 0}^{}(\mathcal{V})}
{\partial p_{\scriptscriptstyle 1}^{2}} \!\right)}\, \circ \, \bar{c}_{\scriptscriptstyle 1, 0}^{}
\end{equation*} 
as claimed. This completes the proof. 
\end{proof}

\vskip10pt
\begin{rem} \!\!It is easy to see that the {\it infinite} homogeneous linear system associated to 
\eqref{eq: system-determine-qn} (with $m \ge 2s + 1 \ge 3$) has only the trivial solution 
if we assume that the unknowns $\bar{c}_{m r,  s}^{}(\mathbf{p})$ 
are subject to a growing condition. Indeed, rewriting \eqref{eq: explicit c1} as 
\begin{equation*}
\prod_{n \ge 1}(1 + p_{\scriptscriptstyle n}^{})^{\!\frac{1}{n}\!\sum_{d \mid n} 
\mu\left(\! \frac{n}{d} \!\right)q^{d}}  =\,  - \, q(q - 1)c_{\scriptscriptstyle 1}^{}(\mathbf{p}) 
+ q^{2} p_{\scriptscriptstyle 1}^{} \! + 1
\end{equation*} 
one finds easily that 
\begin{equation*}  
\psi_{\scriptscriptstyle m}^{}\!\left(\! \frac{\partial c_{\scriptscriptstyle 1}^{}}
{\partial p_{\scriptscriptstyle 1}^{}} \!\right)\!(\mathbf{p}) 
\,= \, \frac{1 + q^{m} c_{\scriptscriptstyle m}^{}(\mathbf{p})  -  q^{m} p_{\scriptscriptstyle m}^{}}
{1 + p_{\scriptscriptstyle m}^{}}
\;\;\;\;\, \text{and}  \;\;\;\;\;  
\psi_{\scriptscriptstyle m}^{}\!\left(\! \frac{\partial c_{\scriptscriptstyle 1}^{}}
{\partial p_{\scriptscriptstyle r}^{}} \!\right)\!(\mathbf{p})
\, = \, \frac{s_{\scriptscriptstyle r}^{}(q^{m}) (q^{m}(q^{m} - 1)c_{\scriptscriptstyle m}^{}(\mathbf{p}) 
-  q^{2 m} p_{\scriptscriptstyle m}^{}  -  1)}
{r q^{m} (q^{m} - 1) (1 + p_{\scriptscriptstyle m r}^{})} \;\;\;\; \text{(if $r\ge 2$).}
\end{equation*} 
Letting $|p_{\scriptscriptstyle n}^{}| < \delta^{n} \slash q^{n}$ ($n \ge 1$), for a fixed $0 < \delta < 1,$ we have that \begin{equation*}
\begin{split}
\left|\psi_{\scriptscriptstyle m}^{}\!\left(\! \frac{\partial c_{\scriptscriptstyle 1}^{}}
{\partial p_{\scriptscriptstyle r}^{}} \!\right)\!(\mathbf{p})\right|
\, & < \, \frac{q^{m r - m}}{r (q^{m} - 1)(1 - q^{- m r})}
(1 + q^{- m} \delta^{m})^{- 1} \cdot  
\prod_{n \ge 1}(1 + q^{- m n} \delta^{m n})^{s_{\scriptscriptstyle n}^{}\!\left(q^{m} \right)\slash n} \\
& = \, \frac{q^{m r - m}}{r (q^{m} - 1)(1 - q^{- m r})} \cdot \frac{1 - q^{- m} \delta^{2 m}}{1 - \delta^{m}}.
\end{split}
\end{equation*} 
Similarly, as long as $\delta$ is not too close to $1$ (for instance, one can take $\delta \le 1 - \frac{3}{q^{2 s + 1} - 1}$), 
we have the lower bound 
\begin{equation*}
\left|\psi_{\scriptscriptstyle m}^{}\!\left(\! \frac{\partial c_{\scriptscriptstyle 1}^{}}
{\partial p_{\scriptscriptstyle 1}^{}} \!\right)\!(\mathbf{p}) \right |
\, >  \frac{q^{m}}{q^{m} - 1}\left(1  - \frac{1 - q^{- m}\delta^{2m}}{q^{m}(1 - \delta^{m})(1 - q^{- m})}\right) \, >  0
\qquad \text{(for $m \ge 2s + 1$ and $r \ge 2$)}.
\end{equation*} 
It follows that, for $|p_{\scriptscriptstyle n}^{}| < \frac{1}{q^{n}}\big(1 - \frac{3}{q^{2 s + 1} - 1}\big)^{n},$ 
we have the estimate 
\begin{equation*} 
\left|\psi_{\scriptscriptstyle m}^{}\!\left(\! \frac{\partial c_{\scriptscriptstyle 1}^{}}
{\partial p_{\scriptscriptstyle r}^{}} \!\right)(\mathbf{p})\slash 
\psi_{\scriptscriptstyle m}^{}\!\left(\! \frac{\partial c_{\scriptscriptstyle 1}^{}}
{\partial p_{\scriptscriptstyle 1}^{}} \!\right)(\mathbf{p}) \right| 
\, < \, \frac{q^{m r - m}}{2 r} 
\qquad \text{($m \ge 2s + 1$ and $r \ge 2$).}
\end{equation*} 
Our assertion follows now -- for instance, from the classical results of von Koch \cite{Koch}; clearly 
this gives rise inductively to the solution $\bar{q}_{\scriptscriptstyle m}^{}(\mathbf{p}, \hbar)$ of 
\eqref{eq: barq-version1} in the proposition.
\end{rem}

\subsection{The asymptotic expansion} 
To obtain the asymptotic expansion of $T(\mathbb{M}\mathcal{V}),$ we first expand 
$\mathcal{K}(\mathbf{p}, \mathbf{q}, \hbar)$ as a power series centered at the unique solution 
$\bar{\mathbf{q}} = (\bar{q}_{\scriptscriptstyle 1}^{}, \, \bar{q}_{\scriptscriptstyle 2}^{}, \ldots)$ of \eqref{eq: barq-version1}. Thus, recalling that 
\begin{equation*} 
\mathcal{K}(\mathbf{p}, \mathbf{q}, \hbar) =  - \sum_{m = 1}^{\infty} 
\left\{\left(q_{\scriptscriptstyle m} \! - p_{\scriptscriptstyle m} \! - \varepsilon_{\scriptscriptstyle m} 
\hbar^{\scriptscriptstyle m \slash 2} \right)^{2} \! \slash 2 m \hbar^{\scriptscriptstyle m} \! \right\} \, +\,  \Psi(T(\mathcal{V}))
\end{equation*} 
we can express 
\begin{equation*} 
\mathcal{K}(\mathbf{p}, \mathbf{q}, \hbar) = \mathcal{K}(\mathbf{p}, \bar{\mathbf{q}}, \hbar) 
\, -  \sum_{m = 1}^{\infty} \frac{t_{\scriptscriptstyle m}^{2}}{2 m \hbar^{m}} 
\, + \, \sum_{|\boldsymbol{\alpha}| \ge 2} 
\partial^{\boldsymbol{\alpha}}\Psi(T(\mathcal{V}))(\bar{\mathbf{q}}, \hbar)
\frac{\text{{\bf t}}^{\boldsymbol{\alpha}}}{\boldsymbol{\alpha}!}
\;\; \qquad \;\;  \text{(with $t_{\scriptscriptstyle m}^{} \! = q_{\scriptscriptstyle m}^{} \! - \bar{q}_{\scriptscriptstyle m}^{}$ 
and $\text{{\bf t}} = {\bf q} - \bar{{\bf q}}$).}
\end{equation*} 
We set 
\begin{equation*} 
\mathcal{L}(\mathbf{p}, \mathbf{q}, \hbar) = \mathcal{L}(\mathbf{p}, \mathrm{{\bf t}}, \hbar) \; : =  
\sum_{|\boldsymbol{\alpha}| \ge 2} 
\partial^{\boldsymbol{\alpha}}\Psi(T(\mathcal{V}))(\bar{\mathbf{q}}, \hbar)
\frac{\mathrm{{\bf t}}^{\boldsymbol{\alpha}}}{\boldsymbol{\alpha}!} \, - \, \frac{1}{2}\sum_{m \ge 1}    
\frac{\partial^{2}\Psi(T(\mathcal{V}))}{\partial q_{\scriptscriptstyle m}^{2}}(\bar{\mathbf{q}}, \hbar) 
\, t_{\scriptscriptstyle m}^{2}
\end{equation*} 
so that  
\begin{equation*} 
\mathcal{K}(\mathbf{p}, \mathbf{q}, \hbar) =  
\mathcal{K}(\mathbf{p}, \bar{\mathbf{q}}, \hbar) 
\, -  \sum_{m = 1}^{\infty}  \left\{\left(1 - m \hbar^{m}
\frac{\partial^{2}\Psi(T(\mathcal{V}))}{\partial q_{\scriptscriptstyle m}^{2}}(\bar{\mathbf{q}}, \hbar)\right) 
\!\frac{t_{\scriptscriptstyle m}^{2}}{2 m \hbar^{m}} \right\} \,+\, \mathcal{L}(\mathbf{p}, \mathbf{q}, \hbar).
\end{equation*}

\vskip10pt
\begin{thm}\label{Main Theorem} --- For coordinates $t_{\scriptscriptstyle 1}^{}, \, t_{\scriptscriptstyle 2}^{}, \ldots$ in 
$\mathrm{Spec}(\Lambda_{\mathrm{alg}} \otimes \mathbb{R}) \cong \mathbb{R}^{\infty}\!,$ set 
\begin{equation*} 
\tilde{t}_{\scriptscriptstyle m}^{} : = \bigg(1 \,-\, m \hbar^{m} \frac{\partial^{2}\Psi(T(\mathcal{V}))}
{\partial q_{\scriptscriptstyle m}^{2}}(\bar{\mathbf{q}}, \hbar) \bigg)^{\! -1\slash 2}
\sqrt{2 m \hbar^{m}}\,  t_{\scriptscriptstyle m}^{}  \;\;\; \mathrm{and} \;\;\; 
\tilde{\mathrm{{\bf t}}} = (\tilde{t}_{\scriptscriptstyle 1}^{}, \, \tilde{t}_{\scriptscriptstyle 2}^{}, \ldots).
\end{equation*} 
If $d\nu(\mathrm{{\bf t}})$ denotes the formal Gaussian measure 
\begin{equation*} 
d\nu(\mathrm{{\bf t}}) = \prod_{m = 1}^{\infty} e^{- t_{\scriptscriptstyle m}^{2}} 
\, \frac{dt_{\scriptscriptstyle m}}{\sqrt{\pi}}
\end{equation*} 
we have 
\begin{equation} \label{eq: main identity}
\Psi(T(\mathbb{M}\mathcal{V})) = \mathcal{K}(\mathbf{p}, \bar{\mathbf{q}}, \hbar) \; - \; 
\frac{1}{2}\sum_{m = 1}^{\infty}\log\left(1 \, - \, m \hbar^{m} \frac{\partial^{2} 
\Psi(T(\mathcal{V}))}{\partial q_{\scriptscriptstyle m}^{2}}(\bar{\mathbf{q}}, \hbar) \right) 
\; + \; \log \left(\int_{\mathbb{R}^{\infty}}^{\! *} e^{\mathcal{L}(\mathbf{p}, \tilde{\mathrm{{\bf t}}}, \hbar)} \,
d\nu(\mathrm{{\bf t}}) \right).
\end{equation} 
Moreover, as $\hbar \rightarrow 0^{+}\!,$ we have the first-order asymptotic expansion 
\begin{equation*} 
\log \left(\int_{\mathbb{R}^{\infty}}^{\! *} e^{\mathcal{L}(\mathbf{p}, \tilde{\mathrm{{\bf t}}}, \hbar)} \, d\nu(\mathrm{{\bf t}}) \right) 
= \left(\frac{\left(\frac{\partial^{2} \mathrm{ch}_{\scriptscriptstyle 0}^{}}
{\partial q_{\scriptscriptstyle 1}^{} \! \partial q_{\scriptscriptstyle 2}^{}}(\bar{\mathbf{q}}) \right)^{\! 2}}
{\left(1 - \frac{\partial^{2} \mathrm{ch}_{\scriptscriptstyle 0}^{}}{\partial q_{\scriptscriptstyle 1}^{2}}(\bar{\mathbf{q}})\right)
\left(1 - \frac{\partial^{2} \psi_{\scriptscriptstyle 2}^{} \left(\mathrm{ch}_{\scriptscriptstyle 0}^{}\right)}
{\partial q_{\scriptscriptstyle 2}^{2}}(\bar{\mathbf{q}}) \right)} \, + \, 
\frac{5}{24}\frac{\left(\frac{\partial^{3} \mathrm{ch}_{\scriptscriptstyle 0}^{}}
{\partial q_{\scriptscriptstyle 1}^{3}}(\bar{\mathbf{q}}) \right)^{\! 2}}
{\left(1 - \frac{\partial^{2} \mathrm{ch}_{\scriptscriptstyle 0}^{}}
{\partial q_{\scriptscriptstyle 1}^{2}}(\bar{\mathbf{q}})\right)^{\! 3}}
\, + \, \frac{1}{8}\frac{\frac{\partial^{4} \mathrm{ch}_{\scriptscriptstyle 0}^{}}
{\partial q_{\scriptscriptstyle 1}^{4}}(\bar{\mathbf{q}})}
{\left(1 - \frac{\partial^{2} \mathrm{ch}_{\scriptscriptstyle 0}^{}}
{\partial q_{\scriptscriptstyle 1}^{2}}(\bar{\mathbf{q}}) \right)^{\! 2}}\! \right)\hbar \, +\, O(\hbar^{2})
\end{equation*} 
where 
$
\mathrm{ch}_{\scriptscriptstyle 0}^{}(\mathbf{q}) = 
\mathrm{ch}_{\scriptscriptstyle 0}^{}(\mathcal{V})(\mathbf{q}) = 
\sum_{n \ge 3} \sum_{\rho \, \vdash \, n}\, \big|\mathscr{M}_{{\scriptscriptstyle 0}, n \scriptscriptstyle \slash \overline{\mathbb{F}}}^{\rho F}\big| 
\frac{q_{\scriptscriptstyle \rho}}{z_{\rho}}.
$
\end{thm}

\begin{proof} \!By Theorem \ref{Theorem GK1}, and Theorem \ref{Theorem GK2} in the form 
\eqref{eq: Getz-Kapr-start-point-integral-representation}, we have 
\begin{equation*} 
\begin{split}
\mathrm{Exp}(T(\mathbb{M}\mathcal{V})) & 
= \exp(\Delta)\mathrm{Exp}(T(\mathcal{V}))\, = 
\int_{\mathbb{R}^{\infty}} \, e^{\mathcal{K}(\mathbf{p},  \mathbf{q},  \hbar)} \, d^{*} \!\mathbf{q} \\
& = e^{\mathcal{K}(\mathbf{p}, \bar{\mathbf{q}}, \hbar)}\!\! \int_{\mathbb{R}^{\infty}} 
e^{\mathcal{L}(\mathbf{p}, \mathrm{{\bf t}}, \hbar)} \exp\!\left[- \sum_{m = 1}^{\infty} \left(1 - m \hbar^{m}
\frac{\partial^{2} \Psi(T(\mathcal{V}))}{\partial q_{\scriptscriptstyle m}^{2}}(\bar{\mathbf{q}}, \hbar)\right) 
\!\frac{t_{\scriptscriptstyle m}^{2}}{2 m \hbar^{m}} \right] \, d^{*} \!\mathrm{{\bf t}}.
\end{split}
\end{equation*} 
Substituting $t_{\scriptscriptstyle m}$ by 
$
{\scriptstyle \big(1 \, - \, m \hbar^{m}
\frac{\partial^{2}\Psi(T(\mathcal{V}))}{\partial q_{\scriptscriptstyle m}^{2}}(\bar{\mathbf{q}}, \hbar) \big)^{\! -1\slash 2}} 
\sqrt{2 m h^{m}} \, t_{\scriptscriptstyle m} 
$ 
for all $m \ge 1,$ and then taking the logarithm, we find that 
\begin{equation*} 
\Psi(T(\mathbb{M}\mathcal{V})) = \mathcal{K}(\mathbf{p}, \bar{\mathbf{q}}, \hbar) \;  - \; 
\frac{1}{2}\sum_{m = 1}^{\infty} \log\left(1 \, - \, m \hbar^{m}\frac{\partial^{2}
\Psi(T(\mathcal{V}))}{\partial q_{\scriptscriptstyle m}^{2}}(\bar{\mathbf{q}}, \hbar) \right) 
\; + \; \log \left(\int_{\mathbb{R}^{\infty}}^{\! *} 
e^{\mathcal{L}(\mathbf{p}, \tilde{\mathrm{{\bf t}}}, \hbar)} \, d\nu(\mathrm{{\bf t}}) \right).
\end{equation*} 
This is the formula for $\Psi(T(\mathbb{M}\mathcal{V}))$ stated in our theorem.

Now let $\boldsymbol{\alpha} = (\alpha_{\text{$\scriptstyle 1$}}, \ldots, \alpha_{s})$ be a multi-index with 
$|\boldsymbol{\alpha}| \ge 2,$ and put 
$\tilde{\boldsymbol{\alpha}} : = (1^{\! \alpha_{\text{$\scriptscriptstyle 1$}}}\!, \ldots, s^{\alpha_{s}}).$ It is clear that when 
computing the partial derivative $\partial^{\boldsymbol{\alpha}}\Psi(T(\mathcal{V}))(\bar{\mathbf{q}}, \hbar),$ the only 
terms 
\begin{equation*} 
\frac{\hbar^{k(g - 1)}}{k}\big|\mathscr{M}_{g, n \scriptscriptstyle \slash \overline{\mathbb{F}}}^{\rho F^{k}}\big| 
\frac{q_{\scriptscriptstyle k \rho}}{z_{\rho}}
\end{equation*} 
in $\Psi(T(\mathcal{V}))$ contributing nontrivially are among those corresponding to $k$'s 
dividing $\tilde{\boldsymbol{\alpha}}$ (i.e., $k$ divides every part of $\tilde{\boldsymbol{\alpha}}$). Using this, 
one finds easily that 
\begin{equation*}
\begin{split}  
\mathcal{L}(\mathbf{p}, \tilde{\mathrm{{\bf t}}}, \hbar) =\, 
& \frac{1}{2 \hbar^{2}}\!\left\{\frac{\partial^{2} (q_{\scriptscriptstyle 2}^{} \! \circ \mathrm{ch}_{\scriptscriptstyle 0}^{}(\mathcal{V}))}{\partial q_{\scriptscriptstyle 2}^{}  \partial q_{\scriptscriptstyle 4}^{}}(\bar{\mathbf{q}})
\tilde{t}_{\scriptscriptstyle 2}^{}\tilde{t}_{\scriptscriptstyle 4}^{}  +  
\frac{\partial^{3} (q_{\scriptscriptstyle 2}^{} \! \circ \mathrm{ch}_{\scriptscriptstyle 0}^{}(\mathcal{V}))}
{\partial q_{\scriptscriptstyle 2}^{3}}(\bar{\mathbf{q}}) \frac{\tilde{t}_{\scriptscriptstyle 2}^{\, 3}}{6} \right\} \, 
+ \, \frac{1}{\hbar}\!\left\{\frac{\partial^{2}\mathrm{ch}_{\scriptscriptstyle 0}^{}(\mathcal{V})}
{\partial q_{\scriptscriptstyle 1}^{} \partial q_{\scriptscriptstyle 2}^{}}
(\bar{\mathbf{q}}){\tilde{t}_{\scriptscriptstyle 1}}^{} {\tilde{t}_{\scriptscriptstyle 2}}^{}  + 
\frac{\partial^{3}\mathrm{ch}_{\scriptscriptstyle 0}^{}(\mathcal{V})}{\partial q_{\scriptscriptstyle 1}^{3}}(\bar{\mathbf{q}})
\frac{{\tilde{t}_{\scriptscriptstyle 1}}^{\, 3}}{6} \right\}\\ 
& \hskip90pt + \, \frac{1}{\hbar}\!\left\{\frac{\partial^{2}\mathrm{ch}_{\scriptscriptstyle 0}^{}(\mathcal{V})}
{\partial q_{\scriptscriptstyle 1}^{} \partial q_{\scriptscriptstyle 3}^{}}(\bar{\mathbf{q}})
{\tilde{t}_{\scriptscriptstyle 1}}^{}  {\tilde{t}_{\scriptscriptstyle 3}}^{} 
+ \frac{\partial^{3} \mathrm{ch}_{\scriptscriptstyle 0}^{}(\mathcal{V})}
{\partial q_{\scriptscriptstyle 1}^{2} \partial q_{\scriptscriptstyle 2}^{}}(\bar{\mathbf{q}})
\frac{{\tilde{t}_{\scriptscriptstyle 1}}^{\, 2} {\tilde{t}_{\scriptscriptstyle 2}}^{}}{2} + 
\frac{\partial^{4}\mathrm{ch}_{\scriptscriptstyle 0}^{}(\mathcal{V})}{\partial q_{\scriptscriptstyle 1}^{4}}(\bar{\mathbf{q}})
\frac{{\tilde{t}_{\scriptscriptstyle 1}}^{\, 4}}{24} \right\} \, +\, O(\hbar^{3 \slash 2});
\end{split}
\end{equation*} 
it is also easy to see that 
\begin{equation} \label{eq: tag basic-asymptotic}
\begin{split}
&\left(1 - m \hbar^{m} \frac{\partial^{2}\Psi(T(\mathcal{V}))}{\partial q_{\scriptscriptstyle m}^{2}}
(\bar{\mathbf{q}}, \hbar) \right)^{\! - 1\slash 2} \\
& \hskip51pt = \left(1 - \frac{\partial^{2} \psi_{\scriptscriptstyle m}^{}(\mathrm{ch}_{\scriptscriptstyle 0}^{}(\mathcal{V}))}
{\partial q_{\scriptscriptstyle m}^{2}}(\bar{\mathbf{q}})\right)^{-1\slash 2}
\left(1 + \frac{\delta_{m, 1}\frac{\partial^{2} \mathrm{ch}_{\scriptscriptstyle 1}^{}\!(\mathcal{V})}
{\partial q_{\scriptscriptstyle 1}^{2}}(\bar{\mathbf{q}}) 
+ 2\delta_{m, 2} \frac{\partial^{2} \mathrm{ch}_{\scriptscriptstyle 0}^{}(\mathcal{V})}
{\partial q_{\scriptscriptstyle 2}^{2}}(\bar{\mathbf{q}})}
{1 - \frac{\partial^{2} \psi_{\scriptscriptstyle m}^{}(\mathrm{ch}_{\scriptscriptstyle 0}^{}(\mathcal{V}))}
{\partial q_{\scriptscriptstyle m}^{2}}(\bar{\mathbf{q}})}\frac{\hbar}{2} \right)\, +\, O(\hbar^{2}).
\end{split}
\end{equation} 
Here $\delta_{m, j}$ is the Kronecker delta. Replacing $\tilde{t}_{\scriptscriptstyle 1}^{},\,  \tilde{t}_{\scriptscriptstyle 2}^{},
\,  \tilde{t}_{\scriptscriptstyle 3}^{}, \ldots$ by their defining expressions and  applying \eqref{eq: tag basic-asymptotic}, 
we obtain an asymptotic expansion of $\mathcal{L}(\mathbf{p}, \tilde{\mathrm{{\bf t}}}, \hbar)$ of the form 
\begin{equation*}
\begin{split}  
\mathcal{L}(\mathbf{p}, \tilde{\mathrm{{\bf t}}}, \hbar) = \, 
(A_{\scriptscriptstyle 1, 2}^{} + B_{\scriptscriptstyle 1, 2}^{} \hbar) \sqrt{\hbar}
\, t_{\scriptscriptstyle 1}^{} t_{\scriptscriptstyle 2}^{} 
& + (A_{\scriptscriptstyle 1^{\! \scriptscriptstyle 3}} + B_{\scriptscriptstyle 1^{\! 3}}\hbar) \sqrt{\hbar}
\, t_{\scriptscriptstyle 1}^{3} + A_{\scriptscriptstyle 2, 4}^{} \hbar 
\,  t_{\scriptscriptstyle 2}^{} t_{\scriptscriptstyle 4}^{}  + 
A_{\scriptscriptstyle 2^{3}} \hbar \,  t_{\scriptscriptstyle 2}^{3} \\
& + A_{\scriptscriptstyle 1, 3}^{} \hbar \, t_{\scriptscriptstyle 1}^{}  t_{\scriptscriptstyle 3}^{}  
+ A_{\scriptscriptstyle 1^{\! 2}\!, 2} 
\hbar\, t_{\scriptscriptstyle 1}^{\, 2} t_{\scriptscriptstyle 2}^{} + 
A_{\scriptscriptstyle 1^{\! 4}} \hbar\, t_{\scriptscriptstyle 1}^{\, 4} \, +\, O(\hbar^{3 \slash 2})
\end{split}
\end{equation*} 
where, for instance, 
\begin{equation*} 
A_{\scriptscriptstyle 1, 2}^{} = 2\sqrt{2}\cdot \frac{\frac{\partial^{2} \mathrm{ch}_{\scriptscriptstyle 0}^{}(\mathcal{V})}
{\partial q_{\scriptscriptstyle 1}^{}\! \partial q_{\scriptscriptstyle 2}^{}}(\bar{\mathbf{q}})}
{\sqrt{\left(1 - \frac{\partial^{2} \mathrm{ch}_{\scriptscriptstyle 0}^{}(\mathcal{V})}
{\partial q_{\scriptscriptstyle 1}^{2}}(\bar{\mathbf{q}})\right)
\left(1 - \frac{\partial^{2} \psi_{\scriptscriptstyle 2}^{}\left(\mathrm{ch}_{\scriptscriptstyle 0}^{}(\mathcal{V})\right)}
{\partial q_{\scriptscriptstyle 2}^{\! 2}}(\bar{\mathbf{q}})\right)}}, \;\;\;\;  
A_{\scriptscriptstyle 1^{\! 3}} = \frac{\sqrt{2}}{3}\cdot \frac{\frac{\partial^{3} \mathrm{ch}_{\scriptscriptstyle 0}^{}(\mathcal{V})}{\partial q_{\scriptscriptstyle 1}^{3}}(\bar{\mathbf{q}})}
{\left(1 - \frac{\partial^{2} \mathrm{ch}_{\scriptscriptstyle 0}^{}(\mathcal{V})}
{\partial q_{\scriptscriptstyle 1}^{2}}(\bar{\mathbf{q}})\right)^{\! 3\slash 2}}
\end{equation*} 
and 
\begin{equation*} 
A_{\scriptscriptstyle 1^{\! 4}} = \frac{1}{6}\cdot \frac{\frac{\partial^{4} \mathrm{ch}_{\scriptscriptstyle 0}^{}(\mathcal{V})}
{\partial q_{\scriptscriptstyle 1}^{4}}(\bar{\mathbf{q}})}
{\left(1 - \frac{\partial^{2} \mathrm{ch}_{\scriptscriptstyle 0}^{}(\mathcal{V})}
{\partial q_{\scriptscriptstyle 1}^{2}}(\bar{\mathbf{q}})\right)^{\! 2}}.
\end{equation*} 
Finally, by applying the familiar Gaussian integral identity
\begin{equation*} 
\int_{\mathbb{R}} t^{k}  e^{- t^{2}} \frac{dt}{\sqrt{\pi}} 
= \begin{cases}  2^{- k\slash 2}(k - 1)!! &\mbox{if $k$ is even}\\
0 & \mbox{if $k$ is odd} 
\end{cases}
\end{equation*} 
one finds easily that 
\begin{equation*} 
\log \left(\int_{\mathbb{R}^{\infty}}^{\! *}  
e^{\mathcal{L}(\mathbf{p}, \tilde{\mathrm{{\bf t}}}, \hbar)} \, d\nu(\mathrm{{\bf t}}) \right) 
\, =\,  \left(\frac{1}{8}A_{\scriptscriptstyle 1, 2}^{2} + \frac{15}{16}A_{\scriptscriptstyle 1^{\! 3}}^{2}
+ \frac{3}{4}A_{\scriptscriptstyle 1^{\! 4}} \! \right)\hbar  +  O(\hbar^{2})
\end{equation*} 
and the asymptotic expansion stated in the theorem follows. This completes the proof. 
\end{proof}

We note that by expressing the exponential $e^{\mathcal{L}(\mathbf{p}, \tilde{\mathrm{{\bf t}}}, \hbar)}$ as a power series 
and then integrating, one can write the full asymptotic expan-\\sion of $\Psi(T(\mathbb{M}\mathcal{V})).$ One should also be able to interpret 
this asymptotic expansion, as in \cite[Proposition~3.6]{BH}, as an expansion over stable graphs. We shall not pursue this further since the 
calculations are quite cumbersome. Instead, we shall just use Theorem \ref{Main Theorem} to obtain formulas for the generating series 
$
\mathrm{ch}_{\scriptscriptstyle g}(\mathbb{M}\mathcal{V})
$ 
when $g \le 2.$

\section{Equivariant Euler characteristics} \label{equid-euler-examples} 
In this section we shall apply Theorem \ref{Main Theorem} to express 
$
\mathrm{ch}_{\scriptscriptstyle g}(\mathbb{M}\mathcal{V}),
$ 
when $g = 0, 1$ and $2,$ in terms of 
$
\mathrm{ch}_{\scriptscriptstyle g'}(\mathcal{V}),
$ 
$g' \le g.$ The formulas for 
$
\mathrm{ch}_{\scriptscriptstyle 0}(\mathbb{M}\mathcal{V})
$ 
and 
$
\mathrm{ch}_{\scriptscriptstyle 1}(\mathbb{M}\mathcal{V})
$ 
are not new; they were first obtained by Getzler and Kapranov \cite{GK} when $g = 0,$ and 
by Getzler \cite{Getz2} when $g = 1.$

From now on we adopt the following notation. Put 
$
\mathrm{\bf  a}_{\scriptscriptstyle g} := \mathrm{ch}_{\scriptscriptstyle g}(\mathcal{V}),\, 
\mathrm{\bf  b}_{\!\scriptscriptstyle g} := \mathrm{ch}_{\scriptscriptstyle g}(\mathbb{M}\mathcal{V}),
$ 
and for $g, \alpha_{\text{$\scriptstyle 1$}}, \ldots, \alpha_{s} \in \mathbb{N}$ ($\alpha_{s} \ge 1$), write
\begin{equation*} 
\mathrm{\bf  a}_{\scriptscriptstyle g}^{\text{($\alpha_{\scriptscriptstyle 1}, \ldots, \alpha_{s}$)}} = 
\frac{\partial^{\text{$\alpha_{\scriptscriptstyle 1} \! +   \cdots  + \alpha_{s}$}} \mathrm{ch}_{\scriptscriptstyle g}(\mathcal{V})}
{\partial q_{\text{$\scriptstyle 1$}}^{\text{$\alpha_{\scriptscriptstyle 1}$}} \cdots \, 
\partial q_{s}^{\text{$\alpha_{s}$}}} \;\;\; \textrm{and} \;\;\;\, 
\mathrm{\bf  b}_{\scriptscriptstyle g}^{\text{($\alpha_{\scriptscriptstyle 1}, \ldots, \alpha_{s}$)}} = 
\frac{\partial^{\text{$\alpha_{\scriptscriptstyle 1} \! +   \cdots  + \alpha_{s}$}} 
\mathrm{ch}_{\scriptscriptstyle g}(\mathbb{M}\mathcal{V})}
{\partial q_{\text{$\scriptstyle 1$}}^{\text{$\alpha_{\scriptscriptstyle 1}$}} \cdots \, 
\partial q_{s}^{\text{$\alpha_{s}$}}}.
\end{equation*} 
In what follows, we shall need the notion of the {\it Legendre transform} for symmetric functions, 
see \cite[Theorem~7.15]{GK}. To define this notion, let $\mathbb{Q}[\!\![x]\!\!]_{*}^{}$ denote the set 
of all formal power series $f \in \mathbb{Q}[\!\![x]\!\!]$ of the form 
\begin{equation*} 
f(x) = \sum_{s = 2}^{\infty} \frac{a_{s} \, x^{\, s}}{s!}  \;\qquad\; \text{(with $a_{2} \ne 0$).}
\end{equation*} 
Let $\mathrm{rk}: \Lambda \to \mathbb{Q}[\!\![x]\!\!]$ be the homomorphism defined 
by $h_{\scriptscriptstyle n} \mapsto x^{n} \slash n!,$ and denote by $\Lambda_{*}$ the set of symmetric functions 
$f \in \Lambda$ such that $\mathrm{rk}(f)\in \mathbb{Q}[\!\![x]\!\!]_{*}^{};$ note that $\mathrm{rk}(f)$ can also be 
obtained from $f(p_{\scriptscriptstyle 1}, p_{\scriptscriptstyle 2}, \ldots)$ by setting: 
$p_{\scriptscriptstyle 1} = x,$ and $p_{\scriptscriptstyle n} = 0$ if $n\ge 2.$ If $f \in \Lambda_{*},$ there is a unique 
element $g = \!\mathscr{L}\!f \in \Lambda_{*},$ called the {\it Legendre transform} of $f,$ determined by the formula 
\begin{equation*} 
g \, \circ \,
\frac{\partial f}{\partial p_{\scriptscriptstyle 1}^{}} + f 
= p_{\scriptscriptstyle 1}^{}\frac{\partial f}{\partial p_{\scriptscriptstyle 1}^{}}. 
\end{equation*} 
With the above notation and terminology, we have (see \cite[Theorem~7.17]{GK} or \cite[Theorem~5.9]{Getz1}):

\vskip10pt
\begin{thm}\label{Getzler1} --- Put 
$
f = e_{\scriptscriptstyle 2}^{} -  \mathrm{\bf  a}_{\scriptscriptstyle 0}^{} 
$ 
and 
$
g = h_{\scriptscriptstyle 2}^{} + \mathrm{\bf  b}_{\scriptscriptstyle 0}^{}. 
$ 
Then 
$
\mathscr{L}\!f = g.
$
\end{thm}

\begin{proof} \!By the definition of the Legendre transform and the identities 
$e_{\scriptscriptstyle 2}^{} = (p_{\scriptscriptstyle 1}^{2} - p_{\scriptscriptstyle 2}^{}) \slash 2$ 
and $h_{\scriptscriptstyle 2}^{} = (p_{\scriptscriptstyle 1}^{2} + p_{\scriptscriptstyle 2}^{}) \slash 2,$ 
we have to show that
\begin{equation*} \label{eq: identity-genus-0}
\bigg(\frac{p_{\scriptscriptstyle 1}^{2}  +  p_{\scriptscriptstyle 2}^{}}{2}  +   
\mathrm{\bf  b}_{\scriptscriptstyle 0}^{} \bigg)
\, \circ \, \left(p_{\scriptscriptstyle 1}^{} - \mathrm{\bf  a}_{\scriptscriptstyle 0}^{\scriptscriptstyle (1)}\right)
= \frac{p_{\scriptscriptstyle 1}^{2} + p_{\scriptscriptstyle 2}^{}}{2}  -  p_{\scriptscriptstyle 1}^{}
\mathrm{\bf  a}_{\scriptscriptstyle 0}^{\scriptscriptstyle (1)}   +  \mathrm{\bf  a}_{\scriptscriptstyle 0}^{}.
\end{equation*} 
To see this, we first observe that, by \eqref{eq: tag basic-asymptotic}, the only piece in the right-hand side 
of \eqref{eq: main identity} contributing negative powers of $\hbar$ to 
$\Psi(T(\mathbb{M}\mathcal{V}))$ is $\mathcal{K}(\mathbf{p}, \bar{\mathbf{q}}, \hbar).$ From the definition of 
$\mathcal{K}(\mathbf{p}, \mathbf{q}, \hbar)$ it is clear that  
\begin{equation*}
\mathcal{K}(\mathbf{p}, \bar{\mathbf{q}}, \hbar) = 
\sum_{m = 1}^{\infty}\left\{ - \left(\bar{q}_{\scriptscriptstyle m}^{} \! - p_{\scriptscriptstyle m}^{} \! - 
\varepsilon_{\scriptscriptstyle m}^{}\hbar^{m \slash 2} \right)^{\! 2}  + \, 
2\psi_{\scriptscriptstyle m}^{}\!\left(\mathrm{\bf  a}_{\scriptscriptstyle 0}^{} \right)
\!(\bar{\mathbf{q}})\right\} \slash 2 m \hbar^{m}   + \, O(1)
\end{equation*} 
with 
$
\psi_{\scriptscriptstyle m}^{}\!\left(\mathrm{\bf  a}_{\scriptscriptstyle 0}^{} \right) 
\!(\bar{\mathbf{q}})
$ 
given explicitly by 
\begin{equation*}
\psi_{\scriptscriptstyle m}^{}\!\left(\mathrm{\bf  a}_{\scriptscriptstyle 0}^{}\right)
\!(\bar{\mathbf{q}}) = 
\sum_{n \ge 3} \sum_{\rho \, \vdash \, n}\,
\big|\mathscr{M}_{{\scriptscriptstyle 0}, n \scriptscriptstyle \slash \overline{\mathbb{F}}}^{\rho F^{m}}\big| 
\frac{\bar{q}_{\scriptscriptstyle m \rho}}{z_{\rho}}. 
\end{equation*} 
Thus we have:
\begin{equation*} 
\Psi(T(\mathbb{M}\mathcal{V}))  = 
\sum_{m = 1}^{\infty}\left\{ - \left(\bar{q}_{\scriptscriptstyle m}^{} \! - p_{\scriptscriptstyle m}^{} \! 
- \varepsilon_{\scriptscriptstyle m}^{}\hbar^{m \slash 2} \right)^{\! 2}  +  
2 \psi_{\scriptscriptstyle m}^{}\!\left(\mathrm{\bf  a}_{\scriptscriptstyle 0}^{} \right) 
\!(\bar{\mathbf{q}})\right\} \slash 2 m \hbar^{m}   + \, O(1).
\end{equation*} 
Now, since 
\begin{equation*} 
\bar{q}_{\scriptscriptstyle r}^{}(\mathbf{p}, \hbar) = \bar{c}_{r, {\scriptscriptstyle 0}}^{}(\mathbf{p})  
+  O\big(\hbar^{(r + 1 - \varepsilon_{r}) \slash 2} \big) 
\qquad  \text{(for all $r\ge 1$)}
\end{equation*}
it follows that for every partition $\rho = \left(1^{\rho(1)}\!, \ldots, n^{\rho(n)} \right),$ we have 
\begin{equation*} 
\bar{q}_{\scriptscriptstyle m \rho}^{}(\mathbf{p}, \hbar)\hbar^{- m} = \left(\bar{c}_{m, {\scriptscriptstyle 0}}^{\, \rho(1)}(\mathbf{p})\hbar^{- m} 
\! + \rho(1)\hbar^{(1 - m - \varepsilon_{\scriptscriptstyle m}) \slash 2} \bar{c}_{m, {\scriptscriptstyle 0}}^{\, \rho(1) - 1}(\mathbf{p})
U_{\! m}^{}(\mathbf{p}, \hbar) \right) \prod_{l = 2}^{n} \bar{c}_{m l, {\scriptscriptstyle 0}}^{\, \rho(l)}(\mathbf{p}) + O(1)
\end{equation*} 
where 
$
U_{\! m}^{}(\mathbf{p}, \hbar) = (\bar{q}_{\scriptscriptstyle m}^{}(\mathbf{p}, \hbar) 
- \bar{c}_{m, {\scriptscriptstyle 0}}^{}(\mathbf{p}))
\hbar^{- (m + 1 - \varepsilon_{\scriptscriptstyle m}) \slash 2}.
$ 
Thus we can replace $\bar{q}_{\scriptscriptstyle m \rho}^{}\hbar^{- m},$ for every $m,$ by 
\begin{equation*} 
\hbar^{ - m}\prod_{l = 1}^{n} \bar{c}_{m l, {\scriptscriptstyle 0}}^{\, \rho(l)} \, + \, 
U_{\! m}^{}\hbar^{(1 - m - \varepsilon_{\scriptscriptstyle m}) \slash 2} \rho(1)\, \bar{c}_{m, {\scriptscriptstyle 0}}^{\, \rho(1) - 1}
\prod_{l = 2}^{n} \bar{c}_{m l, {\scriptscriptstyle 0}}^{\, \rho(l)}.
\end{equation*} 
This immediately yields 
\begin{equation*}
\hbar^{-m}\psi_{\scriptscriptstyle m}^{}\!\left(\mathrm{\bf a}_{\scriptscriptstyle 0}^{}\right)
\!(\bar{\mathbf{q}})  = 
\hbar^{-m}\psi_{\scriptscriptstyle m}^{}\!\left(\mathrm{\bf  a}_{\scriptscriptstyle 0}^{}\right)\!(\bar{C}(\mathbf{p}))  
+ \hbar^{(1 - m - \varepsilon_{\scriptscriptstyle m}) \slash 2} U_{\! m}^{}(\mathbf{p}, \hbar)(\bar{c}_{m, {\scriptscriptstyle 0}}^{}(\mathbf{p}) 
- p_{\scriptscriptstyle m}^{})  +  O(1).
\end{equation*} 
Similarly,
\begin{equation*}
\begin{split}
- \left(\bar{q}_{\scriptscriptstyle m}^{}  \! - p_{\scriptscriptstyle m}^{}  \! - \varepsilon_{\scriptscriptstyle m}^{}\hbar^{m \slash 2} \right)^{\! 2}  
\slash 2  \hbar^{m} & = - \, (\bar{c}_{m, {\scriptscriptstyle 0}}^{}(\mathbf{p}) - p_{\scriptscriptstyle m}^{})^{2} \slash 2 \hbar^{m} +  
\varepsilon_{\scriptscriptstyle m}^{} (\bar{c}_{m, {\scriptscriptstyle 0}}^{}(\mathbf{p}) - p_{\scriptscriptstyle m}^{}) \slash \hbar^{m \slash 2}\\ 
& \hskip20pt  - \hbar^{(1 - m - \varepsilon_{\scriptscriptstyle m}) \slash 2} 
U_{\! m}^{}(\mathbf{p}, \hbar)(\bar{c}_{m, {\scriptscriptstyle 0}}^{}(\mathbf{p}) - p_{\scriptscriptstyle m}^{})  +  O(1).
\end{split}
\end{equation*} 
Putting these calculations together, we find that
\begin{equation*}
\begin{split} 
\Psi(T(\mathbb{M}\mathcal{V})) = 
- \, \frac{1}{2}\Psi\bigg(\frac{(\bar{c}_{\scriptscriptstyle 1, 0}^{}(\mathbf{p}) - p_{\scriptscriptstyle 1}^{})^{2}}{\hbar} \bigg) 
\;\;  & +  \sum_{m-\text{even}} \frac{1}{m}\psi_{\scriptscriptstyle m}^{}
\bigg(\frac{\bar{c}_{\scriptscriptstyle 1, 0}^{}(\mathbf{p}) - p_{\scriptscriptstyle 1}^{}}{\sqrt{\hbar}} \bigg)  \\ 
& +\, \left(\Psi \!\left(\frac{\mathrm{\bf  a}_{\scriptscriptstyle 0}^{}}{\hbar} \right) 
\circ  \bar{c}_{\scriptscriptstyle 1, 0}^{} \right)(\mathbf{p}) \, + \, O(1).
\end{split}
\end{equation*} 
To this asymptotic expansion, we apply the operation $\Psi^{-1}(-) = \sum_{d \ge 1} \frac{\mu(d)}{d}\psi_{d}^{}(-).$ 
Using the identity 
\begin{equation*}
\sum_{\substack{d \mid N \\ \frac{N}{d}-\text{even}}} \mu(d) =
\begin{cases}
1 &\mbox{if} N = 2 \\
0 &\mbox{if} N \ne 2
\end{cases}
\end{equation*} 
and then equating the coefficients of $\hbar^{-1}$ in the resulting asymptotic expansion, one obtains the formula: 
\begin{equation} \label{eq: identity-genus-0-equiv}
\mathrm{\bf b}_{\scriptscriptstyle 0}^{}  
= - \frac{1}{2}(\bar{c}_{\scriptscriptstyle 1, 0}^{} - p_{\scriptscriptstyle 1}^{})^{2} 
- \frac{1}{2}p_{\scriptscriptstyle 2}^{}  + \frac{1}{2}\,
\psi_{\scriptscriptstyle 2}^{}\!\left(\bar{c}_{\scriptscriptstyle 1, 0}^{}\right)
+ \left(\mathrm{\bf  a}_{\scriptscriptstyle 0}^{} 
\circ  \bar{c}_{\scriptscriptstyle 1, 0}^{}\right)\!.
\end{equation} 
The theorem follows now by applying 
$
\circ \, c_{\scriptscriptstyle 1}^{} \, (\text{i.e.,} \,
\circ (p_{\scriptscriptstyle 1}^{} - \, \mathrm{\bf  a}_{\scriptscriptstyle 0}^{\scriptscriptstyle (1)}))
$ 
on the right of \eqref{eq: identity-genus-0-equiv}, and by recalling that 
$
\bar{c}_{\scriptscriptstyle 1, 0}^{}  \circ  c_{\scriptscriptstyle 1}^{} \! = p_{\scriptscriptstyle 1}^{}.
$
\end{proof}

\vskip10pt
\begin{cor} \label{expression of barc1,0} --- The symmetric functions 
\begin{equation*}
p_{\scriptscriptstyle 1}^{} - \, \mathrm{\bf  a}_{\scriptscriptstyle 0}^{\scriptscriptstyle (1)} \;\; \text{and} \;\;\; 
p_{\scriptscriptstyle 1}^{} + \, \mathrm{\bf  b}_{\scriptscriptstyle 0}^{\scriptscriptstyle (1)}
\end{equation*} 
are plethystic inverses. Thus 
$
\bar{c}_{\scriptscriptstyle 1, 0}^{} \! = p_{\scriptscriptstyle 1}^{} + \, 
\mathrm{\bf  b}_{\scriptscriptstyle 0}^{\scriptscriptstyle (1)}.
$ 
\end{cor}

\begin{proof} \!Let $f$ and $g$ be as in Theorem \ref{Getzler1}. Then 
\begin{equation*}
\frac{\partial f}{\partial p_{\scriptscriptstyle 1}^{}} = c_{\scriptscriptstyle 1}^{} \! = p_{\scriptscriptstyle 1}^{} -  \,
\mathrm{\bf  a}_{\scriptscriptstyle 0}^{\scriptscriptstyle (1)} \;\; \text{and} \;\;\; 
\frac{\partial g}{\partial p_{\scriptscriptstyle 1}^{}} = p_{\scriptscriptstyle 1}^{} + \, 
\mathrm{\bf  b}_{\scriptscriptstyle 0}^{\scriptscriptstyle (1)}.
\end{equation*} 
Since 
$
\mathscr{L}\!f = g,
$ 
the first assertion follows at once from \cite[Theorem~7.15 (c)]{GK}. The second assertion is an immediate 
consequence of the fact that $c_{\scriptscriptstyle 1}^{}$ and $\bar{c}_{\scriptscriptstyle 1, 0}^{}$ are 
plethystic inverses. 
\end{proof} 

The following theorem is the main result of \cite{Getz2}; for a direct combinatorial proof, the reader may consult \cite{Peter1}.

\vskip10pt
\begin{thm}\label{Getzler2} --- We have
\begin{equation*}
\mathrm{\bf  b}_{\scriptscriptstyle 1}^{} \! = 
\Bigg\{\mathrm{\bf  a}_{\scriptscriptstyle 1}^{}  
- \, \frac{1}{2}\sum_{m = 1}^{\infty} \frac{\varphi(m)}{m}
\log\left(1 - \psi_{\scriptscriptstyle m}^{}\!\left(\mathrm{\bf a}_{\scriptscriptstyle 0}^{\scriptscriptstyle (2)}\right) \right) \,  + \, 
\frac{\mathrm{\bf a}_{\scriptscriptstyle 0}^{\scriptscriptstyle (0,  1)}\! 
\left(\mathrm{\bf a}_{\scriptscriptstyle 0}^{\scriptscriptstyle (0, 1)}  + \,  1\right)
+ \frac{1}{4}\psi_{\scriptscriptstyle 2}^{}\!\left(\mathrm{\bf a}_{\scriptscriptstyle 0}^{\scriptscriptstyle (2)}\right)}
{1 - \psi_{\scriptscriptstyle 2}^{}\!\left(\mathrm{\bf a}_{\scriptscriptstyle 0}^{\scriptscriptstyle (2)}\right)} \Bigg\} 
\circ  \left(p_{\scriptscriptstyle 1}^{}  + \, \mathrm{\bf b}_{\scriptscriptstyle 0}^{\scriptscriptstyle (1)}\right)\!
\end{equation*} 
where $\varphi(m)$ is Euler's totient function.
\end{thm}

\begin{proof} \!The proof is similar to that of Theorem \ref{Getzler1}. By Theorem \ref{Main Theorem} and the definition of 
$\mathcal{K}(\mathbf{p}, \mathbf{q}, \hbar),$ we have that 
\begin{equation*}
\begin{split}
\Psi(T(\mathbb{M}\mathcal{V})) = 
\Psi(\mathrm{\bf a}_{\scriptscriptstyle 1}^{})(\bar{\mathbf{q}}) \,& + 
\sum_{m = 1}^{\infty} \left\{- \left(\bar{q}_{\scriptscriptstyle m}^{} \! - p_{\scriptscriptstyle m}^{} 
\! - \varepsilon_{\scriptscriptstyle m}^{}\hbar^{m \slash 2} \right)^{\! 2} 
+ \, 2 \psi_{\scriptscriptstyle m}^{}(\mathrm{\bf a}_{\scriptscriptstyle 0}^{})(\bar{\mathbf{q}}) \right\} \slash 2 m \hbar^{m} \\
& - \frac{1}{2}\sum_{m = 1}^{\infty}\log\left(1 - m \hbar^{m}\frac{\partial^{2}
\Psi(T(\mathcal{V}))}{\partial q_{\scriptscriptstyle m}^{2}}(\bar{\mathbf{q}}, \hbar) \right) 
\, + \, O(\hbar).
\end{split}
\end{equation*} 
By Corollary \ref{expression of barc1,0}, the constant term of 
$ 
\Psi(\mathrm{\bf a}_{\scriptscriptstyle 1}^{})(\bar{\mathbf{q}}) 
$ 
is: 
\begin{equation*} 
\Psi(\mathrm{\bf a}_{\scriptscriptstyle 1}^{})(\bar{\mathbf{q}}(\mathbf{p}, 0)) 
= \Psi(\mathrm{\bf a}_{\scriptscriptstyle 1}^{})(\bar{C}(\mathbf{p})) 
= (\Psi(\mathrm{\bf a}_{\scriptscriptstyle 1}^{}) \circ \bar{c}_{\scriptscriptstyle 1, 0}^{})(\mathbf{p})  
= \Psi(\mathrm{\bf a}_{\scriptscriptstyle 1}^{}) \circ 
\left(p_{\scriptscriptstyle 1}^{}  + \, \mathrm{\bf b}_{\scriptscriptstyle 0}^{\scriptscriptstyle (1)}\!(\mathbf{p})\right).
\end{equation*} 
By applying the linear operation $\Psi^{-1}$ to this, we obtain the first contribution to 
$\mathrm{\bf b}_{\scriptscriptstyle 1}^{}$ 
\!stated in the theorem.

The second contribution corresponds to
\begin{equation}  \label{eq: log-part genus 1}
\left. - \frac{1}{2}\sum_{m = 1}^{\infty}\log\left(1 - m \hbar^{m}\frac{\partial^{2}
\Psi(T(\mathcal{V}))}{\partial q_{\scriptscriptstyle m}^{2}}(\bar{\mathbf{q}}, \hbar) \right) \right\vert_{\hbar = 0}.
\end{equation} 
From \eqref{eq: tag basic-asymptotic} and Corollary \ref{expression of barc1,0}, we have 
\begin{equation*} 
\left.\log\left(1 -  m \hbar^{m} \frac{\partial^{2} \Psi(T(\mathcal{V}))}
{\partial q_{\scriptscriptstyle m}^{2}}(\bar{\mathbf{q}}, \hbar) \right)\right\vert_{\hbar = 0}
=\; \log\left(1 - \psi_{\scriptscriptstyle m}^{}\!\left(\mathrm{\bf a}_{\scriptscriptstyle 0}^{\scriptscriptstyle (2)}\right) \right) 
\circ \left(p_{\scriptscriptstyle 1}^{}  + \,  \mathrm{\bf b}_{\scriptscriptstyle 0}^{\scriptscriptstyle (1)}\!(\mathbf{p})\right)
\end{equation*} 
for all $m \ge 1.$ Thus the contribution \eqref{eq: log-part genus 1} is given by  
\begin{equation*}
- \frac{1}{2}\left(\sum_{k = 1}^{\infty} 
\log\left(1 - \psi_{\scriptscriptstyle k}^{}\!\left(\mathrm{\bf a}_{\scriptscriptstyle 0}^{\scriptscriptstyle (2)}\right) \right)\right) 
\circ \left(p_{\scriptscriptstyle 1}^{}  + \, \mathrm{\bf b}_{\scriptscriptstyle 0}^{\scriptscriptstyle (1)}\!(\mathbf{p})\right).
\end{equation*} 
As before, we now apply the operation $\Psi^{-1}(-) = \sum_{r \ge 1} \frac{\mu(r)}{r} \psi_{r}(-).$ Summing first over 
$m = kr,$ and using the well-known M\"obius inversion formula 
\begin{equation*}
\sum_{r\mid m} \frac{\mu(r)}{r} = \frac{\varphi(m)}{m} 
\end{equation*} 
we find that the contribution to $\mathrm{\bf b}_{\scriptscriptstyle 1}^{}$ corresponding to \eqref{eq: log-part genus 1} is  \begin{equation*}
- \frac{1}{2}\left\{\sum_{m = 1}^{\infty} \frac{\varphi(m)}{m}
\log\left(1 - \psi_{\scriptscriptstyle m}^{}\!\left(\mathrm{\bf a}_{\scriptscriptstyle 0}^{\scriptscriptstyle (2)}\right) \right) \right\} 
\circ \left(p_{\scriptscriptstyle 1}^{}  + \, \mathrm{\bf b}_{\scriptscriptstyle 0}^{\scriptscriptstyle (1)}\!(\mathbf{p})\right).
\end{equation*} 
It remains to compute the constant term of 
\begin{equation*} 
\sum_{m = 1}^{\infty} \left\{ - \left(\bar{q}_{\scriptscriptstyle m}^{} \! - p_{\scriptscriptstyle m}^{} 
\! - \varepsilon_{\scriptscriptstyle m}^{}\hbar^{m \slash 2} \right)^{\! 2} 
+ 2\psi_{\scriptscriptstyle m}^{}\!(\mathrm{\bf a}_{\scriptscriptstyle 0}^{})(\bar{\mathbf{q}}) \right\} \slash 2 m \hbar^{m}.
\end{equation*} 
Just as in the proof of Theorem \ref{Getzler1}, for every $m\ge 1,$ the constant term of the summand is 
\begin{equation*} 
\frac{1}{m}\, U_{\scriptscriptstyle 2 m}^{}(\mathbf{p}, 0)\, 
\psi_{\scriptscriptstyle m}^{}\!\left(\mathrm{\bf  a}_{\scriptscriptstyle 0}^{\scriptscriptstyle (0, 1)}\right)(\bar{C}(\mathbf{p})) \, 
+\, \frac{\varepsilon_{m}^{}}{m} \left\{\frac{1}{2}\left( - 1 + \psi_{\scriptscriptstyle m}^{}\!
\left(\mathrm{\bf a}_{\scriptscriptstyle 0}^{\scriptscriptstyle (2)}\right)(\bar{C}(\mathbf{p}))\right)
U_{\scriptscriptstyle m}^{2}(\mathbf{p}, 0) + U_{\scriptscriptstyle m}^{}(\mathbf{p}, 0) - \frac{1}{2} \right\}\!;
\end{equation*} 
by Proposition \ref{unique-formal-sol-barq}, we know that 
\begin{equation*}
U_{\scriptscriptstyle m}^{}(\cdot, 0) = \bar{c}_{\scriptscriptstyle m, m \slash 2}^{} = 
\frac{1  +  2 \psi_{\! \scriptscriptstyle \frac{m}{2}}^{}\!
\left(\mathrm{\bf a}_{\scriptscriptstyle 0}^{\scriptscriptstyle (0, 1)}\right)}
{1 -  \psi_{\scriptscriptstyle m}^{}\!
\left(\mathrm{\bf a}_{\scriptscriptstyle 0}^{\scriptscriptstyle (2)}\right)} 
\circ \bar{c}_{\scriptscriptstyle 1, 0}^{} \;\;\;\;\;\;  \text{(when $m$ is even)}.
\end{equation*} 
Applying the operation $\Psi^{-1}\!,$ it follows that the final contribution to $\mathrm{\bf b}_{\scriptscriptstyle 1}^{}$ is 
\begin{equation*} 
\frac{\left(\mathrm{\bf a}_{\scriptscriptstyle 0}^{\scriptscriptstyle (0, 1)} \right)^{2} \! + \,  
\mathrm{\bf a}_{\scriptscriptstyle 0}^{\scriptscriptstyle (0, 1)} 
+ \frac{1}{4}\psi_{\scriptscriptstyle 2}^{}\!\left(\mathrm{\bf a}_{\scriptscriptstyle 0}^{\scriptscriptstyle (2)}\right)}
{1 - \psi_{\scriptscriptstyle 2}^{}\!\left(\mathrm{\bf a}_{\scriptscriptstyle 0}^{\scriptscriptstyle (2)}\right)} 
\circ  \bar{c}_{\scriptscriptstyle 1, 0}^{}.
\end{equation*} 
The formula of $\mathrm{\bf b}_{\scriptscriptstyle 1}^{}$ \!stated in the theorem follows now from 
Corollary \ref{expression of barc1,0}. 
\end{proof}

To state the analogous result when $g = 2,$ let us introduce some notation. For $m \ge 2$ even, write 
$
\bar{c}_{m, m \slash 2}^{} = \mathrm{\bf v}_{\! \scriptscriptstyle m} \circ \, 
\bar{c}_{\scriptscriptstyle 1, 0}^{},
$ 
where, by Pro-\\position \ref{unique-formal-sol-barq}, 
\begin{equation*}
\mathrm{\bf v}_{\! \scriptscriptstyle m} = \frac{1  +  2 \psi_{\! \scriptscriptstyle \frac{m}{2}}^{}
\!\left(\mathrm{\bf a}_{\scriptscriptstyle 0}^{\scriptscriptstyle (0, 1)} \right)}
{1  -  \psi_{\scriptscriptstyle m}^{}\!\left(\mathrm{\bf  a}_{\scriptscriptstyle 0}^{\scriptscriptstyle (2)}\right)}.
\end{equation*} 
Define 
\begin{equation*} 
\mathrm{\bf  w}_{\scriptscriptstyle 1} = \frac{\mathrm{\bf  a}_{\scriptscriptstyle 0}^{\scriptscriptstyle (1,  1)}
\!\left(1  +  2 \mathrm{\bf  a}_{\scriptscriptstyle 0}^{\scriptscriptstyle (0,   1)} \right)}
{\left(1  -   \mathrm{\bf  a}_{\scriptscriptstyle 0}^{\scriptscriptstyle (2)} \right)
\left(1  -   \psi_{\scriptscriptstyle 2}^{}\!\left(\mathrm{\bf  a}_{\scriptscriptstyle 0}^{\scriptscriptstyle (2)}\right)\right)}  +  
\frac{\mathrm{\bf  a}_{\scriptscriptstyle 1}^{\scriptscriptstyle (1)}}
{1  -   \mathrm{\bf  a}_{\scriptscriptstyle 0}^{\scriptscriptstyle (2)}}
\end{equation*} 
\begin{equation*} 
\begin{split}
\mathrm{\bf  w}_{\scriptscriptstyle 2} =   
\frac{\psi_{\scriptscriptstyle 2}^{}\!\left(\mathrm{\bf  a}_{\scriptscriptstyle 1}^{\scriptscriptstyle (1)} \right)  +  
2  \mathrm{\bf  a}_{\scriptscriptstyle 1}^{\scriptscriptstyle (0,  1)}  +  
2 \mathrm{\bf  a}_{\scriptscriptstyle 0}^{\scriptscriptstyle (1,  1)}\mathrm{\bf  w}_{\! \scriptscriptstyle 1}}
{1  -   \psi_{\scriptscriptstyle 2}^{}\!\left(\mathrm{\bf  a}_{\scriptscriptstyle 0}^{\scriptscriptstyle (2)}\right)}  +  
\frac{2 \mathrm{\bf  a}_{\scriptscriptstyle 0}^{\scriptscriptstyle (0,  2)} \! 
\left(1  +  2 \mathrm{\bf  a}_{\scriptscriptstyle 0}^{\scriptscriptstyle (0,  1)} \right)}
{\left(1  -   \psi_{\scriptscriptstyle 2}^{}\!\left(\mathrm{\bf  a}_{\scriptscriptstyle 0}^{\scriptscriptstyle (2)}\right)\right)^{ 2}}
\, & + \, \frac{1}{2}\frac{\left(1  +  2  \mathrm{\bf  a}_{\scriptscriptstyle 0}^{\scriptscriptstyle (0,  1)} \right)^{ 2}
\! \psi_{\scriptscriptstyle 2}^{}\!\left(\mathrm{\bf  a}_{\scriptscriptstyle 0}^{\scriptscriptstyle (3)}\right)}
{\left(1  -  \psi_{\scriptscriptstyle 2}^{}\!\left(\mathrm{\bf  a}_{\scriptscriptstyle 0}^{\scriptscriptstyle (2)}\right)\right)^{ 3}} \\
& + \, \frac{\left(1  +  2 \psi_{\scriptscriptstyle 2}^{}\!\left(\mathrm{\bf  a}_{\scriptscriptstyle 0}^{\scriptscriptstyle (0,  1)}\right)\right)
\psi_{\scriptscriptstyle 2}^{}\!\left(\mathrm{\bf  a}_{\scriptscriptstyle 0}^{\scriptscriptstyle (1,  1)}\right)}
{\left(1  -   \psi_{\scriptscriptstyle 2}^{}\!\left(\mathrm{\bf  a}_{\scriptscriptstyle 0}^{\scriptscriptstyle (2)}\right)\right)
\left(1  -  \psi_{\scriptscriptstyle 4}^{}\!\left(\mathrm{\bf  a}_{\scriptscriptstyle 0}^{\scriptscriptstyle (2)}\right)\right)}
\end{split}
\end{equation*} 
\begin{equation*} 
\mathrm{\bf  w}_{\scriptscriptstyle 3} = \frac{3  \mathrm{\bf  a}_{\scriptscriptstyle 0}^{\scriptscriptstyle (0,  0,  1)}}
{1  -   \psi_{\scriptscriptstyle 3}^{}\!\left(\mathrm{\bf  a}_{\scriptscriptstyle 0}^{\scriptscriptstyle (2)}\right)}\;\;\;\;\;\;
\mathrm{\bf  w}_{\scriptscriptstyle 4} = \frac{2 \left(1  +  2  \mathrm{\bf  a}_{\scriptscriptstyle 0}^{\scriptscriptstyle (0,  1)} \right)
\psi_{\scriptscriptstyle 2}^{}\!\left(\mathrm{\bf  a}_{\scriptscriptstyle 0}^{\scriptscriptstyle (1,  1)} \right)}
{\left(1  -  \psi_{\scriptscriptstyle 2}^{}\!\left(\mathrm{\bf  a}_{\scriptscriptstyle 0}^{\scriptscriptstyle (2)}\right)\right)
\left(1  -   \psi_{\scriptscriptstyle 4}^{}\!\left(\mathrm{\bf  a}_{\scriptscriptstyle 0}^{\scriptscriptstyle (2)}\right)\right)} \, + \, 
\frac{4 \mathrm{\bf  a}_{\scriptscriptstyle 0}^{\scriptscriptstyle (0,  0,  0,  1)}}
{1  -   \psi_{\scriptscriptstyle 4}^{}\!\left(\mathrm{\bf  a}_{\scriptscriptstyle 0}^{\scriptscriptstyle (2)}\right)}\;\;\;\;\;\; 
\mathrm{\bf  w}_{\scriptscriptstyle 6} =  
\frac{3  \psi_{\scriptscriptstyle 2}^{}\!\left(\mathrm{\bf  a}_{\scriptscriptstyle 0}^{\scriptscriptstyle (0,  0,  1)}\right)}
{1  -   \psi_{\scriptscriptstyle 6}^{}\!\left(\mathrm{\bf  a}_{\scriptscriptstyle 0}^{\scriptscriptstyle (2)}\right)}
\end{equation*} 
and $\mathrm{\bf  w}_{\! \scriptscriptstyle m} \! = 0$ for $m\ge 5, \, m \ne 6.$ By solving for the coefficients 
$
\bar{c}_{m, \, (m + 1 + \varepsilon_{m}) \slash 2}^{}(\mathbf{p})
$ 
in \eqref{eq: system-determine-qn}, we see that 
\begin{equation*}
\bar{c}_{m, \, (m + 1 + \varepsilon_{m}) \slash 2}^{} = \mathrm{\bf  w}_{\! \scriptscriptstyle m} \circ \,\bar{c}_{\scriptscriptstyle 1,  0}^{}.
\end{equation*} 
With these quantities we now define 
\begin{equation}  \label{eq: contribution g = 2}
\begin{split}
&\overline{\mathrm{\bf  b}}_{\scriptscriptstyle 2}^{} : = 
\underbrace{\mathrm{\bf  a}_{\scriptscriptstyle 2}^{}}_{\textrm{\textcolor{MyDarkRed}{${\scriptstyle{g = 2}}$}}}  +\, 
\underbrace{\mathrm{\bf  w}_{\! \scriptscriptstyle 1}^{}
\mathrm{\bf  a}_{\scriptscriptstyle 1}^{\scriptscriptstyle (1)}
\! +\, \mathrm{\bf  v}_{\! \scriptscriptstyle 2}^{}\mathrm{\bf  a}_{\scriptscriptstyle 1}^{\scriptscriptstyle (0,  1)} 
\! + \, \frac{1}{2}\mathrm{\bf  v}_{\! \scriptscriptstyle 2}^{}
\psi_{\scriptscriptstyle 2}^{}\!\left(\mathrm{\bf  a}_{\scriptscriptstyle 1}^{\scriptscriptstyle (1)}\right)
}_{\textrm{\textcolor{MyDarkRed}{${\scriptstyle{g = 1}}$}}} \; + \;
\underbrace{\mathrm{\bf  v}_{\! \scriptscriptstyle 2}^{}\mathrm{\bf  w}_{\! \scriptscriptstyle 1}^{} 
\mathrm{\bf  a}_{\scriptscriptstyle 0}^{\scriptscriptstyle (1,  1)} 
+ \, \frac{1}{2}\mathrm{\bf  v}_{\! \scriptscriptstyle 2}^{ \scriptscriptstyle 2}
\mathrm{\bf  a}_{\scriptscriptstyle 0}^{\scriptscriptstyle (0,  2)}
+ \, \mathrm{\bf  w}_{\! \scriptscriptstyle 3}^{}\mathrm{\bf  a}_{\scriptscriptstyle 0}^{\scriptscriptstyle (0,  0,  1)}
+ \, \mathrm{\bf  v}_{\! \scriptscriptstyle 4}^{}\mathrm{\bf  a}_{\scriptscriptstyle 0}^{\scriptscriptstyle (0,  0,  0,  1)}
}_{\textrm{\textcolor{MyDarkRed}{${\scriptstyle{g  =  0}}$ ${\scriptstyle{\textrm{and}}}$ 
${\scriptstyle{m  = 1}}$}}} \\ 
& \hskip10pt +  \underbrace{ \frac{1}{12}\mathrm{\bf  v}_{\! \scriptscriptstyle 2}^{ \scriptscriptstyle 3}
\psi_{\scriptscriptstyle 2}^{}\!\left(\mathrm{\bf  a}_{\scriptscriptstyle 0}^{\scriptscriptstyle (3)}\right) 
 +  \frac{1}{2}\mathrm{\bf  v}_{\! \scriptscriptstyle 2}^{} \mathrm{\bf  v}_{\! \scriptscriptstyle 4}^{} 
\psi_{\scriptscriptstyle 2}^{}\!\left(\mathrm{\bf  a}_{\scriptscriptstyle 0}^{\scriptscriptstyle (1, 1)}\right) 
 +  \frac{1}{2}\mathrm{\bf  v}_{\! \scriptscriptstyle 6}^{} 
\psi_{\scriptscriptstyle 2}^{}\!\left(\mathrm{\bf  a}_{\scriptscriptstyle 0}^{\scriptscriptstyle (0,  0,  1)}\right)
}_{\textrm{\textcolor{MyDarkRed}{${\scriptstyle{g  =  0}}$ ${\scriptstyle{\textrm{and}}}$ ${\scriptstyle{m  = 2}}$}}}\\   
&  + \! \sum_{m-\textrm{even}} 
\bigg\{\underbrace{\!\! \frac{\mathrm{\bf  v}_{\! \scriptscriptstyle m}^{}\mathrm{\bf  w}_{\! \scriptscriptstyle m}^{}
\psi_{\scriptscriptstyle m}^{}\!\left(\mathrm{\bf  a}_{\scriptscriptstyle 0}^{\scriptscriptstyle (2)}\right)  
+  \mathrm{\bf  w}_{\! \scriptscriptstyle 2 m}^{}
\psi_{\scriptscriptstyle m}^{}\!\left(\mathrm{\bf  a}_{\scriptscriptstyle 0}^{\scriptscriptstyle (0,  1)}\right)}{m}
}_{\textrm{\textcolor{MyDarkRed}{${\scriptstyle{g  =  0, \, m  \le 6}}$}}} \, + \!\!\!\!\!\!\!\!\!
\underbrace{\frac{\mathrm{\bf  w}_{\! \scriptscriptstyle m}^{}\left(1 -  \mathrm{\bf  v}_{\! \scriptscriptstyle m}^{}\right)}{m}
}_{\textrm{\textcolor{MyDarkRed}
{${\scriptstyle{ - \, \left(\bar{q}_{\scriptscriptstyle m}^{} - \, p_{\scriptscriptstyle m}^{} - \, \hbar^{m \slash 2} \right)^{ 2} 
\! \slash 2 m \hbar^{m}\!, \,  m \le 6}}$}}} \!\!\!\!\!\!\!\!\!\!\!\! \bigg\}
\; + \!\! \sum_{m-\textrm{odd}} 
\bigg\{\underbrace{\!\! \frac{\mathrm{\bf  w}_{\! \scriptscriptstyle 2 m}^{}
\psi_{\scriptscriptstyle m}^{}\!\left(\mathrm{\bf  a}_{\scriptscriptstyle 0}^{\scriptscriptstyle (0,  1)}\right) 
 + \frac{1}{2}\mathrm{\bf  w}_{\! \scriptscriptstyle m}^{ \scriptscriptstyle 2}
\psi_{\scriptscriptstyle m}^{}\!\left(\mathrm{\bf  a}_{\scriptscriptstyle 0}^{\scriptscriptstyle (2)}\right)}{m}
}_{\textrm{\textcolor{MyDarkRed}
{${\scriptstyle{g = 0, \,  m \le 3}}$}}} \, - \!\!\!\!\!\!\!\!\!\!\!\!\!\!\!\!\!\!\!
\underbrace{\frac{\mathrm{\bf  w}_{\! \scriptscriptstyle m}^{ \scriptscriptstyle 2}}{2 m} 
}_{\textrm{\textcolor{MyDarkRed}
{${\scriptstyle{\!\! - \, \left(\bar{q}_{\scriptscriptstyle m}^{} - \,  p_{\scriptscriptstyle m}^{} \right)^{ 2} 
\! \slash 2 m \hbar^{m}\!, \,  m \le 3}}$}}}  \!\!\!\!\!\!\!\!\!\!\!\!\!\!\!\!\!\!\!\!\bigg\} \\ 
&  \hskip10pt +  \underbrace{\frac{1}{2}\left(\frac{\mathrm{\bf  w}_{\! \scriptscriptstyle 1}^{}
\mathrm{\bf  a}_{\scriptscriptstyle 0}^{\scriptscriptstyle (3)}  \! +  \mathrm{\bf  v}_{\! \scriptscriptstyle 2}^{}
\mathrm{\bf  a}_{\scriptscriptstyle 0}^{\scriptscriptstyle (2,  1)}}{1  -  \mathrm{\bf a}_{\scriptscriptstyle 0}^{\scriptscriptstyle (2)}}  +  
\frac{\mathrm{\bf  v}_{\! \scriptscriptstyle 2}^{} 
\psi_{\scriptscriptstyle 2}^{}\!\left(\mathrm{\bf  a}_{\scriptscriptstyle 0}^{\scriptscriptstyle (3)}\right)}
{1  - \psi_{\scriptscriptstyle 2}^{}\!\left(\mathrm{\bf  a}_{\scriptscriptstyle 0}^{\scriptscriptstyle (2)}\right)}
 +  \frac{\mathrm{\bf  a}_{\scriptscriptstyle 1}^{\scriptscriptstyle (2)}}
{1  -  \mathrm{\bf  a}_{\scriptscriptstyle 0}^{\scriptscriptstyle (2)}}\right)
\; + \, \frac{\mathrm{\bf  a}_{\scriptscriptstyle 0}^{\scriptscriptstyle (0,  2)}}
{1  -  \psi_{\scriptscriptstyle 2}^{}\!\left(\mathrm{\bf  a}_{\scriptscriptstyle 0}^{\scriptscriptstyle (2)} \right)}}_{\textrm{\textcolor{MyDarkRed}
{${\scriptstyle{- \frac{1}{2}\sum_{m = 1}^{\infty}\log\left(1  -  m \hbar^{m}\frac{\partial^{2}
\Psi(T(\mathcal{V}))}{\partial q_{\scriptscriptstyle m}^{2}}(\bar{\mathbf{q}}, \hbar) \right)}}$}}}\\
&  \hskip10pt +  \underbrace{\frac{\left(\mathrm{\bf  a}_{\scriptscriptstyle 0}^{\scriptscriptstyle (1,  1)}\right)^{\! 2}}
{\left(1 - \mathrm{\bf  a}_{\scriptscriptstyle 0}^{\scriptscriptstyle (2)}\right)
\left(1 -  \psi_{\scriptscriptstyle 2}^{}\!\left(\mathrm{\bf  a}_{\scriptscriptstyle 0}^{\scriptscriptstyle (2)}\right)\right)} + 
\frac{5}{24}\frac{\left(\mathrm{\bf  a}_{\scriptscriptstyle 0}^{\scriptscriptstyle (3)}\right)^{\! 2}}
{\left(1 - \mathrm{\bf  a}_{\scriptscriptstyle 0}^{\scriptscriptstyle (2)}\right)^{\! 3}}
 +  \frac{1}{8}\frac{\mathrm{\bf  a}_{\scriptscriptstyle 0}^{\scriptscriptstyle (4)}}
{\left(1 - \mathrm{\bf  a}_{\scriptscriptstyle 0}^{\scriptscriptstyle (2)}\right)^{\! 2}}}_{\textrm{\textcolor{MyDarkRed}
{$\log \left(\int_{\mathbb{R}^{\infty}}^{\! *} \!e^{\mathcal{L}(\mathbf{p}, \tilde{\mathrm{{\bf t}}}, \hbar)} \, d\nu(\mathrm{{\bf t}}) \right) $}}}. 
\end{split}
\end{equation}

\vskip10pt
\begin{thm}\label{Genus2equiv-Euler-ch} --- With the above notation, we have
$
\mathrm{\bf  b}_{\scriptscriptstyle 2}^{} \!= \overline{\mathrm{\bf  b}}_{\scriptscriptstyle 2}^{} \circ  
\left(p_{\scriptscriptstyle 1}^{} + \, \mathrm{\bf  b}_{\scriptscriptstyle 0}^{\scriptscriptstyle (1)}\right).
$
\end{thm}

\begin{proof} \!The assertion follows, as before, from Theorem \ref{Main Theorem} by computing the coefficient of $\hbar$ 
in the right-hand side of \eqref{eq: main identity}\footnote{The operation $\Psi^{-1}$ does not have any effect in this case.}. 
Each contribution to this coefficient corresponds to a piece of 
$\overline{\mathrm{\bf  b}}_{\scriptscriptstyle 2}^{},$ as indicated in \eqref{eq: contribution g = 2}. For instance, let us 
verify the contribution to $\mathrm{\bf  b}_{\scriptscriptstyle 2}^{}$ coming from 
\begin{equation*}
- \frac{1}{2}\sum_{m = 1}^{\infty}\log\left(1  -  m \hbar^{m}\frac{\partial^{2}
\Psi(T(\mathcal{V}))}{\partial q_{\scriptscriptstyle m}^{2}}(\bar{\mathbf{q}}, \hbar) \right). 
\end{equation*} 
By \eqref{eq: tag basic-asymptotic}, it suffices to compute the coefficient of $\hbar$ in
\begin{equation} \label{eq: log-genus-2-contrib-equiv}
-  \frac{1}{2}\sum_{m = 1}^{\infty}
\log\left(1 - \psi_{\scriptscriptstyle m}^{}\!
\left(\mathrm{\bf  a}_{\scriptscriptstyle 0}^{\scriptscriptstyle (2)}\right)\!(\bar{\mathbf{q}})\right)
\; +\;  \frac{\mathrm{\bf  a}_{\scriptscriptstyle 1}^{\scriptscriptstyle (2)}(\bar{\mathbf{q}})}
{1 - \mathrm{\bf  a}_{\scriptscriptstyle 0}^{\scriptscriptstyle (2)}(\bar{\mathbf{q}})} \cdot \frac{\hbar}{2}
\; + \, \frac{\mathrm{\bf  a}_{\scriptscriptstyle 0}^{\scriptscriptstyle (0,  2)}(\bar{\mathbf{q}})}
{1 - \psi_{\scriptscriptstyle 2}^{}\!\left(\mathrm{\bf  a}_{\scriptscriptstyle 0}^{\scriptscriptstyle (2)}\right)\!(\bar{\mathbf{q}})} 
\cdot \hbar.
\end{equation} 
The logarithmic terms do not contribute to the coefficient when $m\ge 3.$ If $m = 1$ we have 
\begin{equation*} 
\mathrm{\bf  a}_{\scriptscriptstyle 0}^{\scriptscriptstyle (2)}(\bar{\mathbf{q}})  =  
\left(\mathrm{\bf  a}_{\scriptscriptstyle 0}^{\scriptscriptstyle (2)} \circ \, \bar{c}_{\scriptscriptstyle 1,  0}^{}\right)\!(\mathbf{p})
 \, + \, \hbar\left(\left(\mathrm{\bf  w}_{\! \scriptscriptstyle 1}^{}\mathrm{\bf  a}_{\scriptscriptstyle 0}^{\scriptscriptstyle (3)} 
+\,  \mathrm{\bf  v}_{\! \scriptscriptstyle 2}^{}\mathrm{\bf  a}_{\scriptscriptstyle 0}^{\scriptscriptstyle (2,  1)}\right) 
\circ  \bar{c}_{\scriptscriptstyle 1,  0}^{}\right)\!(\mathbf{p})\, +\, O(\hbar^{2})
\end{equation*} 
and 
\begin{equation*} 
\psi_{\scriptscriptstyle 2}^{}\!\left(\mathrm{\bf  a}_{\scriptscriptstyle 0}^{\scriptscriptstyle (2)}\right)\!(\bar{\mathbf{q}}) =  
\left(\psi_{\scriptscriptstyle 2}^{}\!\left(\mathrm{\bf  a}_{\scriptscriptstyle 0}^{\scriptscriptstyle (2)}\right) \circ  
\bar{c}_{\scriptscriptstyle 1,  0}^{}\right)\!(\mathbf{p})
\, + \, \hbar\left(\mathrm{\bf  v}_{\! \scriptscriptstyle 2}^{}\psi_{\scriptscriptstyle 2}^{}
\!\left(\mathrm{\bf  a}_{\scriptscriptstyle 0}^{\scriptscriptstyle (3)}\right)
\circ  \bar{c}_{\scriptscriptstyle 1,  0}^{}\right)\!(\mathbf{p})\, +\, O(\hbar^{2})
\end{equation*} 
if $m = 2.$ Thus the coefficient of $\hbar$ in \eqref{eq: log-genus-2-contrib-equiv} is indeed 
\begin{equation*}
\left(\frac{1}{2}\left(\frac{\mathrm{\bf  w}_{\! \scriptscriptstyle 1}^{}
\mathrm{\bf  a}_{\scriptscriptstyle 0}^{\scriptscriptstyle (3)} \! +  
\mathrm{\bf  v}_{\! \scriptscriptstyle 2}^{}
\mathrm{\bf  a}_{\scriptscriptstyle 0}^{\scriptscriptstyle (2,  1)}}
{1 -  \mathrm{\bf a}_{\scriptscriptstyle 0}^{\scriptscriptstyle (2)}}  +  
\frac{\mathrm{\bf  v}_{\! \scriptscriptstyle 2}^{} 
\psi_{\scriptscriptstyle 2}^{}\!\left(\mathrm{\bf  a}_{\scriptscriptstyle 0}^{\scriptscriptstyle (3)}\right)}
{1  -   \psi_{\scriptscriptstyle 2}^{}\!\left(\mathrm{\bf  a}_{\scriptscriptstyle 0}^{\scriptscriptstyle (2)}\right)}
 +  \frac{\mathrm{\bf  a}_{\scriptscriptstyle 1}^{\scriptscriptstyle (2)}}
{1 -  \mathrm{\bf  a}_{\scriptscriptstyle 0}^{\scriptscriptstyle (2)}}\right)
\; + \, \frac{\mathrm{\bf  a}_{\scriptscriptstyle 0}^{\scriptscriptstyle (0,  2)}}
{1 -  \psi_{\scriptscriptstyle 2}^{}\!\left(\mathrm{\bf  a}_{\scriptscriptstyle 0}^{\scriptscriptstyle (2)}\right)}\right)
\,\circ \, \bar{c}_{\scriptscriptstyle 1,  0}^{}
\end{equation*} 
(evaluated at $\mathbf{p}$); we recall that 
$
\bar{c}_{\scriptscriptstyle 1,  0}^{} 
= p_{\scriptscriptstyle 1}^{} \!+ \mathrm{\bf  b}_{\scriptscriptstyle 0}^{\scriptscriptstyle (1)}. 
$ 

All the other contributions are computed similarly. 
\end{proof}

We note that using the above results, the generating series $\mathrm{\bf  b}_{\scriptscriptstyle 2}^{}$ can be effectively computed. Indeed, since 
$
\mathrm{\bf  a}_{\scriptscriptstyle 0}^{} = \mathrm{ch}_{\scriptscriptstyle 0}(\mathcal{V})
$ 
is known (see the remark preceding Proposition \ref{unique-formal-sol-barq}), it suffices to compute the generating series 
$
\mathrm{\bf  a}_{\scriptscriptstyle 1}^{}
$ 
and 
$
\mathrm{\bf  a}_{\scriptscriptstyle 2}^{}.
$ 
An expression in closed form for 
$
\mathrm{\bf  a}_{\scriptscriptstyle 1}^{}
$ 
can be read off directly from Getzler's formula \cite[Eq.~(5.5)]{Getz0} for the generating series of the $\mathbb{S}_{n}$-equivariant 
Serre characteristic of $\mathscr{M}_{1,  n}.$ To compute 
$
\mathrm{\bf  a}_{\scriptscriptstyle 2}^{},
$ 
let $\mathscr{A}_{2}$ denote the moduli space of principally polarized abelian surfaces. Via the Torelli map, we can view 
$\mathscr{M}_{2}$ as an open substack of $\mathscr{A}_{2},$ and set $\mathscr{A}_{1,1} = \mathscr{A}_{2} \setminus \mathscr{M}_{2}.$ 
For $\lambda = (\lambda_{1} \ge \lambda_{2} \ge 0),$ we have (see \cite{FC}) natural $\ell$-adic smooth \'etale sheaves 
$\mathbb{V}_{\lambda}$ on $\mathscr{A}_{2} \otimes \mathbb{Z}[1\slash \ell]$ corresponding to irreducible algebraic representations 
of $\mathrm{GSp}_{4}(\mathbb{Q}).$ From the results of \cite{Getz0} (or rather their $\ell$-adic realization), we know that computing 
the $\mathbb{S}_{n}$-equivariant Euler characteristic 
$ 
\mathrm{ch}_{\scriptscriptstyle n}(\mathcal{V}(\!\!(2, n)\!\!)\!) 
$ 
amounts to the same as computing the Euler characteristics 
\begin{equation*}
\mathbf{e}_{\scriptscriptstyle c}(\mathscr{X}, \mathbb{V}_{\lambda}) = 
\sum_{i}\, (- 1)^{i}[H^{\scriptscriptstyle i}_{\scriptscriptstyle c}
(\mathscr{X}_{\slash \overline{\mathbb{Q}}}, \mathbb{V}_{\lambda})]
\end{equation*} 
for $\mathscr{X} = \mathscr{A}_{2}$ and $\mathscr{X} = \mathscr{A}_{1,1},$ and all $\lambda$ with $\lambda_{1} + \lambda_{2} \le n.$ 
A formula for $\mathbf{e}_{\scriptscriptstyle c}(\mathscr{A}_{2}, \mathbb{V}_{\lambda})$ was conjectured by Faber and van der Geer in 
\cite{FvdG}. Their conjecture was proved by Weissauer \cite{Weiss} when $\lambda$ is regular, that is, $\lambda_{1} > \lambda_{2} > 0,$ 
and by Petersen \cite{Peter2} for arbitrary $\lambda.$ The characteristic $\mathbf{e}_{\scriptscriptstyle c}(\mathscr{A}_{1,1}, \mathbb{V}_{\lambda})$ 
can be computed, for instance, using the branching formula in \cite[Section 3]{Peter3}. Thus taking the trace of Frobenius one obtains 
a formula expressing $\mathrm{\bf  a}_{\scriptscriptstyle 2}^{}$ in terms of traces of Hecke operators on spaces of elliptic and genus $2$ 
vector-valued Siegel modular forms.

%


\end{document}